\documentclass[12pt,psamsfonts]{amsart}
\usepackage{fullpage,amssymb,amsfonts,amsmath}
\usepackage[all,cmtip]{xy}
\usepackage{enumerate}
\usepackage{mathrsfs}
\usepackage{comment}
\usepackage{hyperref} 
\usepackage[capitalise]{cleveref}
\usepackage{floatrow}
\usepackage{todonotes}
\usepackage{url}
\usepackage{constants}

\let\oldtocsection=\tocsection

\let\oldtocsubsection=\tocsubsection

\let\oldtocsubsubsection=\tocsubsubsection

\renewcommand{\tocsection}[2]{\hspace{0em}\oldtocsection{#1}{#2}}
\renewcommand{\tocsubsection}[2]{\hspace{1em}\oldtocsubsection{#1}{#2}}
\renewcommand{\tocsubsubsection}[2]{\hspace{2em}\oldtocsubsubsection{#1}{#2}}
\newtheorem{thm}{Theorem}[section]
\newtheorem{cor}[thm]{Corollary}
\newtheorem{prop}[thm]{Proposition}
\newtheorem{lem}[thm]{Lemma}

\newtheorem{quest*}{Question}
\newtheorem{prob*}{Problem}

\theoremstyle{definition}

\theoremstyle{remark}

\numberwithin{equation}{section}
\newcounter{notation}

\DeclareUrlCommand\DOI{}
\newcommand{\doi}[1]{\href{http://dx.doi.org/#1}{\DOI{doi:#1}}}



\crefname{figure}{Figure}{Figures}
\theoremstyle{plain}
\newtheorem*{thm*}{Theorem}
\crefname{thm}{Theorem}{Theorems}
\crefname{cor}{Corollary}{Corollarys}
\newtheorem*{cor*}{Corollary}
\crefname{cor*}{Corollary}{Corollarys}
\crefname{lem}{Lemma}{Lemmas}
\crefname{prop}{Proposition}{Propositions}
\crefname{conj}{Conjecture}{Conjectures}
\newtheorem*{conj*}{Conjecture}
\crefname{conj*}{Conjecture}{Conjectures}
\crefname{defn}{Definition}{Definitions}

\theoremstyle{remark}
\newtheorem*{remark}{Remark}
\newtheorem*{remarks}{Remarks}
\newtheorem*{rem*}{Remark}

\def\addsymbol #1: #2#3{$#1$ \> \parbox{5.4in}{#2 \dotfill \pageref{#3}}\\} 
\def\addsymbolEND #1: #2#3{$#1$ \> \parbox{5.4in}{#2 \dotfill \pageref{#3}}}

\newcommand{\ds}{\displaystyle}

\renewcommand{\bar}{\overline}

\newcommand{\kd}{\mathfrak{d}}

\newcommand{\kf}{\mathfrak{f}}

\newcommand{\cL}{\mathcal{L}}

\renewcommand{\pmod}[1]{\, (\mathrm{mod} {\, #1})}
\newcommand{\kn}{\mathfrak{n}}
\newcommand{\N}{\mathrm{N}}

\newcommand{\kp}{\mathfrak{p}}

\newcommand{\cO}{\mathcal{O}}
\newcommand{\ord}{\mathrm{ord}\,}

\newcommand{\kq}{\mathfrak{q}}
\newcommand{\Q}{\mathbb{Q}}

\newcommand{\R}{\mathbb{R}}

\renewcommand{\Re}{\mathrm{Re}}
\def\Res{\mathop{\mathrm{Res}}}

\newcommand{\Z}{\mathbb{Z}}
\newcommand{\cZ}{\mathcal{Z}}

\newconstantfamily{abcon}{symbol=c}

\title{Explicit results on the distribution of zeros of Hecke $L$-functions}

\author{Jesse Thorner}
\email{jesse.thorner@gmail.com}

\author{Asif Zaman}
\thanks{The second author was supported in part by an NSERC PGS-D scholarship.} 
\email{asif@math.toronto.edu}

\date{\today}

\begin{document}

\maketitle
\begin{abstract}
We prove an explicit log-free zero density estimate and an explicit version of the zero-repulsion phenomenon of Deuring and Heilbronn for Hecke $L$-functions.  In forthcoming work of the second author, these estimates will be used to establish explicit bounds on the least norm of a prime ideal in a congruence class group and improve upon existing explicit bounds for the least norm of a prime ideal in the Chebotarev density theorem.
\end{abstract}

\section{Introduction and Statement of Results}

In 1837, Dirichlet proved that if $a,q\in\mathbb{Z}$ and $(a,q)=1$, then there are infinitely many primes $p\equiv a\pmod q$.  In light of this result, it is natural to ask how big is the first such prime, say $P(a,q)$?  Assuming the Generalized Riemann Hypothesis (GRH) for Dirichlet $L$-functions, Lamzouri, Li, and Soundararajan \cite{LLS} proved that if $q\geq4$, then
\begin{equation}
\label{eqn:Linnik_GRH}
P(a,q)\leq(\varphi(q)\log q)^2,
\end{equation}
where $\varphi$ is Euler's totient function.  Nontrivial, unconditional upper bounds are significantly harder to prove.  The first such bound on $P(a,q)$ is due to Linnik \cite{Linnik}, who proved that for some absolute constant $\Cr{c1}>0$, we have that 
\begin{equation}
\label{eqn:Linnik}
P(a,q)\ll q^{\Cl[abcon]{c1}}
\end{equation}
with an absolute implied constant.  Admissible values of $\Cr{c1}$ are now known explicitly, with the current record being $\Cr{c1}=5.2$ due to Xylouris \cite{Xylouris}. For a detailed history, see Section 1 of Heath-Brown \cite{HBLinnik} and the sources contained therein.

In order to obtain small values of $\Cr{c1}$, one typically requires three principles; for example, the following explicit forms of these principles are found in \cite[Section 1]{HBLinnik}:
\begin{itemize}
\item A zero-free region for Dirichlet $L$-functions \cite{Chen}:  if $q$ is sufficiently large, then the product $\prod_{\chi\bmod q}L(s,\chi)$ has at most one zero in the region
\begin{equation}
\label{eqn:ZFR_Dirichlet}
s=\sigma+it,\qquad \sigma\geq1-\frac{0.10367}{\log(q(2+|t|))}.
\end{equation}
If such an exceptional zero exists, then it is real and simple and it corresponds with a non-trivial real character $\chi$.
\item A ``log-free'' zero density estimate \cite{Huxley, Jutila}:   If $q$ is sufficiently large, $\epsilon>0$, and we define
\[
N(\sigma,T,\chi)=\#\{\rho=\beta+i\gamma:L(\rho,\chi)=0,|\gamma|\leq T,\beta\geq\sigma\},
\]
then
\begin{equation}
\label{eqn:LFZDE_Dirichlet}
\sum_{\chi\bmod q}N(\sigma,T,\chi)\ll (qT)^{(\frac{12}{5}+\epsilon)(1-\sigma)},\qquad T\geq1,
\end{equation}
where the implied constant depends on $\epsilon$.
\item The zero repulsion phenomenon of Deuring and Heilbronn \cite[Chapter 10]{Graham1}:  if $q$ is sufficiently large, $\lambda>0$ is sufficiently small, $\epsilon>0$, and the exceptional zero in the region \eqref{eqn:ZFR_Dirichlet} exists and equals $1-\lambda/\log q$, then $\prod_{\chi\pmod q}L(s,\chi)$ has no other zeros in the region
\begin{equation}
\label{eqn:DH_Dirichlet}
\sigma\geq1-\frac{(\frac{2}{3}-\epsilon)(\log \lambda^{-1})}{\log(q(2+|t|))}.
\end{equation}
\end{itemize}

Weiss \cite{Weiss} considered a generalization of \eqref{eqn:Linnik} in the context of a general number field.  Let $K/\Q$ be a number field with absolute field norm $\mathrm{N}$ and absolute discriminant $D_K$, and let $\mathfrak{q}$ be an integral ideal of $K$.  One considers the (narrow) ray class group  $I(\mathfrak{q})/P_{\mathfrak{q}}$ where $I(\mathfrak{q})$ is the group of fractional ideals of $K$ which are coprime to $\mathfrak{q}$ and $P_{\mathfrak{q}}$ is the subgroup of principal ideals $(\alpha)$ with $\alpha$ totally positive and $\alpha \equiv 1 \pmod{\kq}$.  Let $H$ be a subgroup of $I(\kq)$ containing $P_{\kq}$; we call any such subgroup a congruence class group of $K$.  Weiss proved that there exist absolute constants $\Cl[abcon]{Weiss_1}>0$ and $\Cl[abcon]{Weiss_2}>0$ such that each coset of $H$ in $I(\kq)$ contains a prime ideal $\kp$ satisfying
\begin{equation}
\label{eqn:Weiss1}
\mathrm{N}\mathfrak{p}\leq 2[K:\Q]^{\Cr{Weiss_1}[K:\Q]}(D_K\mathrm{N}\mathfrak{q})^{\Cr{Weiss_2}}.
\end{equation}
Consequently, each ideal class of $K$ contains a prime ideal $\mathfrak{p}$ satisfying
\[
\mathrm{N}\mathfrak{p}\leq 2[K:\Q]^{\Cr{Weiss_1}[K:\Q]}D_K^{\Cr{Weiss_2}}.
\]
To prove \eqref{eqn:Weiss1}, Weiss proved variants of \eqref{eqn:ZFR_Dirichlet}-\eqref{eqn:DH_Dirichlet} for Hecke $L$-functions with completely effective field uniformity.

An even broader generalization of \eqref{eqn:Linnik} lies in the context of the Chebotarev density theorem.  Let $L/F$ be a Galois extension of number fields with Galois group $G$.  To each unramified prime ideal $\mathfrak{p}$ of $F$, there corresponds a certain conjugacy class of Frobenius automorphisms in $G$ which are attached to the prime ideals of $L$ lying above $\mathfrak{p}$.  We denote this conjugacy class using the Artin symbol $[\frac{L/F}{\mathfrak{p}}]$.  For a fixed conjugacy class $C\subset G$, let
\[
\pi_C(x):=\#\Big\{\mathfrak{p}:\textup{$\mathfrak{p}$ is unramified, $\Big[\frac{L/F}{\mathfrak{p}}\Big]=C$, $\mathrm{N}\mathfrak{p}\leq x$}\Big\},
\]
where $\mathrm{N}=\mathrm{N}_{F/\mathbb{Q}}$ is the absolute norm of $F$.  The Chebotarev density theorem asserts that
\[
\pi_C(x)\sim\frac{|C|}{|G|}\mathrm{Li}(x).
\]
In analogy with \eqref{eqn:Linnik}, it is natural to bound the quantity
\[
P(C,L/F):=\min \Big\{ \mathrm{N}\mathfrak{p} : \textup{$\mathfrak{p}$ is unramified, $\Big[\frac{L/F}{\mathfrak{p}}\Big]=C$, $\mathrm{N}\mathfrak{p}$ is a rational prime} \Big\}. 
\]
Under GRH for Hecke $L$-functions, Bach and Sorenson \cite{BS} proved that
\begin{equation}
\label{eqn:CDT_GRH}
P(C,L/F)\leq (4\log D_L+2.5 [L:\mathbb{Q}]+5)^2.
\end{equation}
We note that if $L=\mathbb{Q}(e^{2\pi i/q})$ for some integer $q\geq3$ and $F=\mathbb{Q}$, then we recover a bound of the same analytic quality as \eqref{eqn:Linnik_GRH}, though the constants are a bit larger.

The first nontrivial, unconditional bound on $P(C,L/F)$ is due to Lagarias, Montgomery, and Odlyzko \cite{LMO}; they proved that for some absolute constant $\Cl[abcon]{c2}>0$, we have that
\begin{equation}
\label{eqn:LMO}
P(C,L/F)\leq 2D_L^{\Cr{c2}}
\end{equation}
Equation \eqref{eqn:LMO} (up to the computation of $\Cr{c2}$) is commensurate with the best known bounds when $L=\mathbb{Q}(\sqrt{D})$ for some fundamental discriminant $D$, $F=\mathbb{Q}$, and $C$ is the nontrivial conjugacy class of $G$, in which case we are measuring the least quadratic nonresidue modulo $D$ (see Burgess \cite{Burgess}).  
Recently, the second author \cite{Zaman_2015c} proved that one may take $\Cr{c2}=40$ for $D_L$ sufficiently large.  We observe, however, that if $L=\mathbb{Q}(e^{2\pi i/q})$ and $F=\mathbb{Q}$, then \eqref{eqn:LMO} is exponential in $q$, which is significantly worse than \eqref{eqn:Linnik}.

To explain how \eqref{eqn:Weiss1} relates to this Chebotarev setting, we must establish some notation.  Let $A$ be any abelian subgroup of $G$ such that $A\cap C$ is nonempty, let $\widehat{A}$ be the character group of $A$, let $\kf_{\chi} = \mathfrak{f}(\chi)$ be the conductor of a character $\chi\in\widehat{A}$, let $\mathrm{Ind}_A^G \chi$ be a character of $\widehat{G}$ induced by $\chi\in\widehat{A}$, and let
\[
\mathcal{Q}=\mathcal{Q}(C,L/F;A):=\max\Big\{D_F^{[L:F]/|A|}\mathrm{N}\mathfrak{f}(\mathrm{Ind}_A^G\chi):\textup{$\chi\in\widehat{A}$ irreducible}\Big\}.
\]
Using the fundamental theorem of class field theory, Deuring's trick \cite{Deuring}, and \eqref{eqn:Weiss1}, Weiss \cite[Theorem 6.1]{Weiss} proved that for certain absolute constants $\Cl[abcon]{c3}>0$ and $\Cl[abcon]{c4}>0$, 
\begin{equation}
\label{eqn:Weiss}
P(C,L/F)\leq 2[L:\mathbb{Q}]^{\Cr{c3}[L:\mathbb{Q}]/|A|}\mathcal{Q}^{\Cr{c4}}.
\end{equation}
When $A$ is cyclic, we have from the conductor-discriminant formula that
\[
D_L^{1/|A|}\leq\mathcal{Q}\leq D_L^{1/\varphi(|A|)}.
\]
(See \cite[Chapter 5, Section 3]{Weiss_Thesis} for a proof of the upper bound.)  Thus Weiss proves a bound on $P(C,L/F)$ which provides a ``continuous transition'' from \eqref{eqn:Linnik} to \eqref{eqn:LMO} with the potential to create significant savings over \eqref{eqn:LMO} when $G$ has a large abelian subgroup which intersects $C$.  In particular, if $L$ is a cyclotomic extension of $F = \mathbb{Q}$, then \eqref{eqn:Linnik} and \eqref{eqn:Weiss} are equivalent.


The fundamental difference between \eqref{eqn:LMO} and \eqref{eqn:Weiss} is that the proof of \eqref{eqn:LMO} does not take full advantage of the factorization of the Dedekind zeta function $\zeta_L(s)$ of $L$ into a product of Hecke $L$-functions; this choice affords one the opportunity to use more elementary tools.  The proof of \eqref{eqn:Weiss1}, and hence the proof of \eqref{eqn:Weiss},  takes advantage of the factorization of $\zeta_L(s)$, which requires the use of a log-free zero density estimate as in Linnik's original work.

Our goal in this paper is to prove explicit versions of Weiss' field-uniform variants of \eqref{eqn:LFZDE_Dirichlet} and \eqref{eqn:DH_Dirichlet}.  In a forthcoming paper, the second author \cite{Zaman3} will employ these explicit results to make $\Cr{Weiss_1}, \Cr{Weiss_2}, \Cr{c3}$, and $\Cr{c4}$ explicit.  We note that Fogels \cite{Fogels} was the first to prove variants of Principles 2 and 3 for Hecke characters, though his proof did not maintain the necessary field uniformity.  Weiss' results rely critically on his field-uniform variants of Fogels' work, but Weiss' results are not explicit.


In \cref{sec:LFZD_Proof}, we prove an explicit version of Weiss' variant of \eqref{eqn:LFZDE_Dirichlet} for Hecke characters \cite[Corollary 4.4]{Weiss}.  To state Weiss' result, we first introduce some notation.  Let $H \pmod{\kq}$ be a congruence class group of $K$ (that is, $H$ is a subgroup of $I(\kq)$ containing $P_{\kq}$), let $n_K=[K:\Q]$, and let $h_H=[I(\kq):H]$.  Define
\begin{equation}
	Q = Q_H := \max\{ \N\kf_{\chi} : \chi \pmod{\kq} \text{ satisfying } \chi(H) = 1\},
	\label{def:Q}
\end{equation}
and
\[
N(\sigma,T,\chi) := \#\{ \rho = \beta+i\gamma : L(\rho, \chi) = 0, \sigma < \beta < 1, |\gamma| \leq T \}
\]
where the nontrivial zeros $\rho$ of $L(s,\chi)$ are counted with multiplicity.  Weiss \cite[Corollary 4.4]{Weiss} proved that there exists an absolute constant $\Cl[abcon]{Weiss2}>0$ such that if $\tfrac{1}{2} \leq \sigma < 1$ and $T \geq n_K^2 h_H^{1/n_K}$, then
\begin{equation}
\sum_{\substack{\chi\pmod{\kq}\\ \chi(H) = 1}} N(\sigma, T,\chi) \ll (e^{O(n_K)}D_K^2 Q T^{n_K})^{\Cr{Weiss2}(1-\sigma)}
\label{eqn:WeissDensity}
\end{equation}
with an absolute and computable implied constant. The first main result of this paper exhibits an explicit value of $c_7$.

\begin{thm}
\label{LFZD-MainTheorem}
Let $H \pmod{\kq}$ be a congruence class group of $K$. Let $n_K=[K:\Q]$ and $Q$ be as in \eqref{def:Q}.  If $\tfrac{1}{2} \leq \sigma < 1$ and $T \geq \max\{  n_K D_K^{-2/n_K} Q^{-3/5n_K}, 1\}$, then 
\begin{equation}
\sum_{\substack{\chi\pmod{\kq} \\ \chi(H) = 1}} N(\sigma, T,\chi) \ll \{ e^{O(n_K)} D_K^{2} Q T^{n_K}\}^{81(1-\sigma)} 
\label{eqn:LFZD-MainTheorem}
\end{equation}
where all implied constants are absolute and computable. If $1- 10^{-3} \leq\sigma<1$, then one may replace 81 with 74.
\end{thm}

\begin{remarks} $ $ 
\begin{itemize}
	\item \cref{LFZD-MainTheorem} also contains a noticeable improvement over Weiss' density estimate \eqref{eqn:WeissDensity} in the range of $T$. One would expect in many applications that the number field $K$ satisfies $n_K^{n_K} \ll D_K^2 Q^{3/5}$, in which case \cref{LFZD-MainTheorem} holds for $T \geq 1$. Even for arbitrary $K$, this will result in appreciable numerical savings in the computation of $\Cr{Weiss_1}$, $\Cr{Weiss_2}$, $\Cr{c3}$, and $\Cr{c4}$ in \cite{Zaman3} instead of simply following Weiss' original arguments \cite[Sections 5-6]{Weiss}.
	\item The appearance of $e^{O(n_K)}$ in \eqref{eqn:LFZD-MainTheorem} may seem unusual for an explicit result but it is always a negligible term. If $n_K = o(\log D_KQ)$ for a certain family of number fields $K$ then
		\[
		e^{O(n_K)} D_K^{2} Q T^{n_K} = D_K^{2+o(1)} Q^{1+o(1)} T^{n_K}
		\]
		so we may ignore the contribution of $e^{O(n_K)}$. Recall a classical bound of Minkowski implies $n_K = O(\log D_K)$ so the above scenario is often the case. Otherwise, if $n_K \gg \log (D_KQ)$ then \eqref{eqn:LFZD-MainTheorem} holds for $T \gg n_K$ in which case 
		\[
		e^{O(n_K)} D_K^{2} Q T^{n_K} = T^{\{1+o(1)\}n_K}
		\]
		so we may again ignore $e^{O(n_K)}$. 
\end{itemize}
\end{remarks} 


We prove Theorem \ref{LFZD-MainTheorem} by constructing a Dirichlet polynomial which is bounded away from zero when in close proximity to a nontrivial zero of a Hecke $L$-function.  This is ensured by using the Tur\'an power sum method (cf. \cref{prop:ZeroDetector}).  The contributions from the detected zeros are summed efficiently using a large sieve inequality for Hecke characters (cf. \cref{thm:LargeSieve}).  In order to maintain desirable field uniformity in our large sieve inequality, we use the Selberg sieve instead of the usual duality arguments; see \cref{sec:MeanValue} for a more detailed discussion.

In order to bound sums over integral ideals, we are required to smooth the sums using a kernel which is ${n_K}$-times differentiable, where $n_K=[K:\Q]$.  Unfortunately, the smoothing introduces the powers of ${n_K}^{{n_K}}$  (see the comments immediately preceding \cite[Section 1]{Weiss}).  We note that if ${n_K}$ is small in comparison to $\log D_K / \log\log D_K$ (i.e., if the root discriminant of $K$ is large), then the powers of ${n_K}^{{n_K}}$ may be safely absorbed into the powers of $D_K$.  On the other hand, if ${n_K}$ is large in comparison to $\log D_K / \log\log D_K$ (i.e., if the root discriminant of $K$ is small), then ${n_K}^{{n_K}}$ dominates $D_K$; this rare situation happens, for example, when considering the infinite $p$-class tower extensions studied by Golod and \v{S}afarevi\v{c} \cite{GS}.

We also note that in the case of bounding the least prime in an arithmetic progression, Tur\'an's power sum method does not produce the strongest numerical results.  Instead, one typically constructs a suitable mollifier for Dirichlet $L$-functions relies on cancellation arising from the M{\"o}bius function.  However, relying on M{\"o}bius cancellation for Hecke $L$-functions introduces dependence on $D_K$ in the implied constant of Theorem \ref{LFZD-MainTheorem}, which is catastrophic for bounds for the least prime ideal in a congruence class.  To the authors' knowledge, the only device by which one can detect zeros to prove a log-free zero density estimate while maintaining suitable field uniformity is the Tur\'an power sum.  (The Tur\'an power sum method was recently used by Lemke Oliver and the first author \cite{RJLO_JT} to prove an effective log-free zero density estimate for Rankin-Selberg $L$-functions.  Since uniformity in certain parameters was crucial for applications, the Tur\'an power sum method was used there as well.)

In \cref{sec:DH_proof}, we prove an explicit variant of the zero repulsion phenomenon of Deuring and Heilbronn for Hecke $L$-functions.

\begin{thm} \label{DH-MainTheorem} Let $\psi \pmod{\kq}$ be a real Hecke character  and suppose $L(s,\psi)$ has a real zero $\beta_1$. Let $T \geq 1$ be given, and $\chi \pmod{\kq}$ be an arbitrary Hecke character and let $\rho' = \beta'+i\gamma'$ be a zero of $L(s,\chi)$ satisfying 
\begin{equation}
\frac{1}{2} \leq \beta' < 1, \qquad |\gamma'| \leq T. 
\label{DH-MainTheorem_BetaRange}
\end{equation}
Then 
\[
\beta' \leq 1 - \frac{ \log\Big( \dfrac{c}{(1-\beta_1) \log(D_K \cdot \N\kq \cdot (T+20)^{n_K} \cdot e^{n_K} ) } \Big) }{a_1  \log D_K + a_2 \log \N\kq  + a_3 n_K \log(T+20)  + a_4 n_K + 10}
\]
for some absolute, computable constant $c > 0$ and 
\[
(a_1,a_2,a_3,a_4) = 
\begin{cases}
(51, 54, 26, 74)  & \text{if $\psi$ is quadratic}, \\
(26, 13, 13, 37) & \text{if $\psi$ is trivial.}
\end{cases}
\]
\end{thm}  
\begin{remarks} $ $
\begin{itemize}
	\item 	 Let $\epsilon>0$.  If we consider a sequence of number fields $K$ in which $n_K = o(\log D_K \N\kq)$ then one may take
\[
(a_1,a_2,a_3,a_4) =
\begin{cases}
(48+\epsilon, 48+\epsilon, 24+\epsilon, 0)  & \text{if $\psi$ is quadratic}, \\
(24+\epsilon, 12+\epsilon, 12+\epsilon, 0) & \text{if $\psi$ is trivial}
\end{cases}
\]
when $D_K \N\kq$ is sufficiently large in terms of $\epsilon$. (See the remark at the end of \cref{DH-MainTheorem_Quadratic} for details.) 	
	\item One may take $\kq$ to be the least common multiple of the conductors $\kf_{\chi}$ and $\kf_{\psi}$. 
\end{itemize}

\end{remarks}


The proof of \cref{DH-MainTheorem} is inspired by \cite[Theorem 5.1]{LMO} and its quantitative counterpart \cite[Theorem 1.2]{Zaman_2015c}. Namely, we apply a power sum inequality and carefully estimate various sums over zeros of Hecke $L$-functions. Other quantitative versions of Deuring-Heilbronn phenomenon have been established by Kadiri and Ng \cite{KadiriNg} and the second author \cite{Zaman_2015c} for the zeros of the Dedekind zeta function and by the second author \cite{Zaman_2015a}  for the zeros of Hecke $L$-functions. The results found in \cite{KadiriNg,Zaman_2015a} use completely different methods than those used here and have much better explicit constants but, instead of assuming \eqref{DH-MainTheorem_BetaRange}, one must restrict to an asymptotically smaller range of $\beta'$ and $|\gamma'| \leq 1$. In other words, the key difference between \cref{DH-MainTheorem} and the aforementioned results is the wide range of validity given by \eqref{DH-MainTheorem_BetaRange}. Consequently, if the real character $\psi$ has a real zero $\beta_1$ exceptionally close to $1$ (often referred to as a Siegel zero), then \cref{DH-MainTheorem} allows one to take full advantage of the repulsion effect.

The paper proceeds as follows.  In \cref{HeckeLFunctions}, we introduce the relevant notation and  conventions, review some standard results in the theory of Hecke $L$-functions, prove some explicit estimates involving Hecke $L$-functions, and bound some standard arithmetic sums over integral ideals of $K$.  In \cref{sec:LFZD_Proof}, we prove \cref{LFZD-MainTheorem} under the assumption of \cref{thm:LargeSieve} (which we prove in \cref{sec:MeanValue}) and \cref{prop:ZeroDetector} (which we prove in \cref{sec:ZeroDetector}).  In \cref{sec:DH_proof}, we prove \cref{DH-MainTheorem}.

\subsection*{Acknowlegements}
The authors thank John Friedlander and Robert Lemke Oliver for their comments and suggestions.  The first author conducted work on this paper while visiting Centre de Recherches Math\'ematiques (hosted by Andrew Granville, Chantal David, and Dimitris Koukoulopoulos) and Stanford University (hosted by Kannan Soundararajan and Robert Lemke Oliver); he is grateful to these departments and hosts for providing a rich and productive work environment.

\section{Auxiliary Estimates}
\label{HeckeLFunctions} 

\subsection{Notation} 
\label{subsec:notation}
We will use the following notation throughout the paper:
\begin{itemize}
	\item $K$ is a number field.
	\item $\mathcal{O}_K$ is the ring of integers of $K$.
	\item $n_K = [K:\mathbb{Q}]$ is the degree of $K/\mathbb{Q}$.
	\item $D_K$ is the absolute value of the discriminant of $K$.
	\item $\N = \N^K_{\Q}$ is the absolute field norm of $K$. 
	\item $\zeta_K(s)$ is the Dedekind zeta function of $K$.
	\item $\mathfrak{q}$ is an integral ideal of $K$.
	\item $\mathrm{Cl}(\mathfrak{q}) = I(\mathfrak{q})/P_{\mathfrak{q}}$ is the narrow ray class group of $K$ modulo $\mathfrak{q}$.
	\item $\chi$, or $\chi \pmod{\kq}$, is a character of $\mathrm{Cl}(\kq)$, referred to as a Hecke character or ray class character of $K$.
	\item $\delta(\chi)$ is the indicator function of the trivial character.
	\item $\kf_{\chi}$ is the conductor of $\chi$; that is, it is the maximal integral ideal such that $\chi$ is induced from a primitive character $\chi^* \pmod{\kf_{\chi}}$.
	\item $D_{\chi} = D_K \N\kf_{\chi}$.
	\item $L(s,\chi)$ is the Hecke $L$-function associated to $\chi$.
	\item $H$, or $H \pmod{\kq}$, is a subgroup of $\mathrm{Cl}(\mathfrak{q})$, or equivalently of $I(\mathfrak{q})$ containing $P_{\mathfrak{q}}$. The group $H$ is referred to as a congruence class group of $K$. 
 	\item $Q = Q_H = \max\{ \N\kf_{\chi} : \chi \pmod{\kq} \text{ satisfying } \chi(H) = 1\}$ is the maximum analytic conductor of $H$. 
	\item $\kf_H = \mathrm{lcm}\{ \kf_{\chi} : \chi \pmod{\kq} \text{ satisfying } \chi(H) = 1\}$ is the conductor of $H$.
	\item $H^* \pmod{\kf_H}$ is the ``primitive" congruence class group inducing $H$.
	\item $h_H = [I(\kq):H]$.   
\end{itemize}
We also adhere to the convention that all implied constants in all asymptotic inequalities $f\ll g$ or $f=O(g)$ are absolute with respect to $K$.  If an implied constant depends on a field-independent parameter, such as $\epsilon$, then we use $\ll_{\epsilon}$ and $O_{\epsilon}$ to denote that the implied constant depends at most on $\epsilon$.  All implied constants will be effectively computable.


\subsection{Hecke $L$-functions}
For a more detailed reference on Hecke $L$-functions, see \cite{LO,Heilbronn} for example. Strictly speaking, a Hecke character $\chi$ is a function on $\mathrm{Cl}(\kq)$ but, by pulling back the domain of $\chi$ and extending it by zero, we regard $\chi$ as a function on integral ideals of $K$.  We will use this convention throughout the paper. 

For the entirety of this section, assume that $\chi$ is primitive. The \emph{Hecke $L$-function of $\chi$}, denoted $L(s, \chi)$, is defined as 
\begin{equation}
L(s,\chi) = \sum_{\kn} \chi(\kn) \N\kn^{-s} = \prod_{\kp} \Big(1-\frac{\chi(\kp)}{\N\kp^{s} } \Big)^{-1}
\label{def:Hecke_L-fcn}
\end{equation}
for $\Re\{s\} > 1$ where the sum is over integral ideals $\kn$ of $K$ and the product is over prime ideals $\kp$ of $K$. Define the \emph{completed Hecke $L$-function} $\xi(s, \chi)$ by
\begin{equation}
\xi(s, \chi) = \big[ s(s-1) \big]^{\delta(\chi)} D_{\chi}^{s/2} \gamma_{\chi}(s) L(s, \chi),
\label{def:CompletedHecke}
\end{equation}
where $D_{\chi} = D_K \N\kf_{\chi}$, $\delta(\chi)$ equals $1$ if $\chi$ is trivial and $0$ otherwise, and $\gamma_{\chi}(s)$ is the \emph{gamma factor of $\chi$} defined by
\begin{equation}
\gamma_{\chi}(s) =  \Big[ \pi^{-\tfrac{s}{2}} \Gamma\Big(\frac{s}{2}\Big)  \Big]^{a(\chi)} \cdot \Big[ \pi^{-\tfrac{s+1}{2} } \Gamma\Big( \frac{s+1}{2} \Big)   \Big]^{b(\chi)}.
\label{GammaFactor} 
\end{equation}
Here $a(\chi)$ and $b(\chi)$ are certain non-negative integers satisfying 
\begin{equation}
a(\chi) + b(\chi) = n_K. 
\label{GammaFactor_Exponents}
\end{equation}
It is a classical fact that $\xi(s, \chi)$ is entire of order 1 and satisfies the functional equation
\begin{equation}
\xi(s, \chi) = w(\chi) \xi(1-s, \bar{\chi})
\label{FunctionalEquation}
\end{equation}
where $w(\chi) \in \mathbb{C}$ is the \emph{root number} of $\chi$ satisfying $|w(\chi)| = 1$. 
The zeros of $\xi(s,\chi)$ are the \emph{non-trivial zeros $\rho$} of $L(s,\chi)$, which satisfy $0 < \Re\{\rho\} < 1$.  The \emph{trivial zeros $\omega$} of $L(s, \chi)$ are given by 
\begin{equation}
\mathop{\ord}_{s = \omega} L(s, \chi) = 
\begin{cases}
a(\chi) - \delta(\chi) & \text{if } \omega = 0, \\
b(\chi) &  \text{if } \omega = -1,-3,-5,\dots\\
a(\chi) & \text{if } \omega = -2,-4,-6, \dots
\end{cases}
\label{TrivialZeros}
\end{equation}
and arise as poles of the gamma factor of $L(s,\chi)$. 

Since $\xi(s,\chi)$ is entire of order 1, it admits a Hadamard product factorization given by
\begin{equation}
\label{eqn:Hadamard}
\xi(s,\chi)=e^{A(\chi)+B(\chi)s}\prod_{\rho }\left(1-\frac{s}{\rho}\right)e^{s/\rho}.
\end{equation}
The zeros $\rho$ of $\xi(s,\chi)$ are the non-trivial zeros of $L(s,\chi)$ and are known to satisfy $0 < \Re\{\rho\} < 1$. We now collect some standard results on $L(s,\chi)$ which follow from Theorems 5.6 and Proposition 5.7 of \cite{IK}.

\begin{lem} \label{ExplicitFormula} Let $\chi$ be a primitive Hecke character. Then
\[
- \Re\Big\{\frac{L'}{L}(s,\chi)\Big\} = \frac{1}{2} \log D_{\chi} +\Re\Big\{\frac{\delta(\chi)}{s-1}+ \frac{\delta(\chi)}{s}\Big\} -  \sum_{\rho }\Re\Big\{ \frac{1}{s-\rho}\Big\}  + \Re\Big\{\frac{\gamma_{\chi}'}{\gamma_{\chi}}(s)\Big\}.
\]
where the sum is over all non-trivial zeros $\rho$ of $L(s,\chi)$. 
\end{lem}
\begin{proof} See \cite[Lemma 5.1]{LO} for example. 
\end{proof}
By similar arguments, there exists an explicit formula for higher derivatives of $-\frac{L'}{L}(s,\chi)$. 
\begin{lem} \label{ExplicitFormula_HigherDerivatives}
Let $\chi$ be a Hecke character (not necessarily primitive) and $k \geq 1$ be a positive integer. Then
\begin{align*}
\frac{(-1)^{k+1}}{k!} \frac{d^k}{ds^k} \frac{L'}{L}(s,\chi) & = 
\frac{1}{k!} \sum_{\kp} \sum_{m=1}^{\infty} (\log \N\kp) (\log \N\kp^m)^k \chi(\kp) (\N\kp)^{-ms}  \\
&  = \frac{\delta(\chi)}{(s-1)^{k+1}}  -  \sum_{\omega} \frac{1}{(s-\omega)^{k+1}}
\end{align*}
for $\Re\{s\} > 1$, where the first sum is over prime ideals $\kp$ of $K$ and the second sum is over all zeros $\omega$ of $L(s,\chi)$, including trivial ones, counted with multiplicity. 
\end{lem}
\begin{proof}  
Using the Hadamard product \eqref{eqn:Hadamard} of $\xi(s,\chi)$, it follows that
\[
(s-1)^{\delta(\chi)} L(s,\chi) = s^r e^{m_1 +m_2s} \prod_{\omega \neq 0} \Big(1-\frac{s}{\omega} \Big) e^{s/\omega}
\] 
where $m_1,m_2$ are constants depending on $\chi$, the product is over all zeros $\omega \neq 0$ of $L(s,\chi)$, including trivial ones, and $r = \ds\mathop{\ord}_{s= 0} L(s,\chi)$. 
 Taking the logarithmic derivative of both sides yields
\[
- \frac{L'}{L}(s,\chi) = \frac{\delta(\chi)}{s-1} - m_2  - \sum_{\omega \neq 0} \Big( \frac{1}{s-\omega} + \frac{1}{\omega} \Big)  - \frac{r}{s}. 
\]
On the other hand, the Euler product of $L(s,\chi)$ implies
\[
- \frac{L'}{L}(s,\chi) = \sum_{\kp} \sum_{m=1}^{\infty} (\log \N\kp)   \chi(\kp) (\N\kp)^{-ms} \quad \text{for $\Re\{s\} > 1$}. 
\]
Differentiating $k$ times both of these formulas for $ - \frac{L'}{L}(s,\chi)$  and multiplying by $(-1)^k/k!$ yields the desired result. Note that the final sum over zeros $\omega$ of $L(s,\chi)$ includes $\omega = 0$, if it exists. 
\end{proof}

\subsection{Explicit $L$-function estimates}

In order to obtain explicit results, we must have explicit bounds on a few important quantities.  First, we record a bound for $L(s,\chi)$ in the critical strip $0 < \Re\{s\} < 1$ via a Phragmen-Lindel\"{o}f type convexity estimate due to Rademacher. 

\begin{lem}[Rademacher \cite{Rademacher}]
\label{Rademacher}
Let $\chi$ be a primitive Hecke character and $\eta \in (0, 1/2]$. Then for $s = \sigma+it$,
\[
|L(s,\chi)| \ll \Big| \frac{1+s}{1-s}\Big|^{\delta(\chi)} \zeta_{\Q}(1+\eta)^{n_K} \Big(  \frac{ D_{\chi} }{ (2\pi)^{n_K}} (3+|t|)^{n_K} \Big)^{ (1+\eta-\sigma)/2} 
\]
uniformly in the strip $-\eta \leq \sigma \leq 1+\eta$. 
\end{lem}

Next, we record an explicit bound on the digamma function and $\frac{\gamma_{\chi}'}{\gamma_\chi}(s)$. 
\begin{lem}
\label{digamma}
Let $s = \sigma+it$ with $\sigma > 1$ and $t \in \R$. Then
\[
\Re\left\{ \frac{\Gamma'}{\Gamma}(s)\right\}\leq \log|s| + \sigma^{-1}
\]
and, for any Hecke character $\chi$, 
\[
\Re\left\{\frac{\gamma_{\chi}'}{\gamma_{\chi}}(s)\right\}\leq \frac{n_K}{2}\left(\log(|s|+1) + \sigma^{-1} -\log \pi\right).
\]
\end{lem}
\begin{proof} The first estimate follows from \cite[Lemma 4]{OS}. The second estimate is a straightforward consequence of the first combined with the definition of $\gamma_{\chi}(s)$ in \eqref{GammaFactor}. 
\end{proof}

Next, we establish some bounds on the number of zeros of $L(s,\chi)$ in a circle. 

\begin{lem} \label{ZerosInCircle-Classical}
Let $\chi$ be a Hecke character. Let $s=\sigma+it$ with $\sigma > 1$ and $t \in \R$. For $r > 0$, denote
\[
N_{\chi}(r; s) :=\#\{\rho=\beta+i\gamma: 0 < \beta < 1, L(\rho, \chi)=0,|s-\rho|\leq r\},
\]
then, for $0 < r \leq 1$, 
\[
N_{\chi}(r; s) \leq \{ 4\log D_K + 2 \log \N \kf_{\chi} +2n_K\log(|t|+3) + 2n_K + 4 + 4\delta(\chi) \} \cdot r + 4 + 4\delta(\chi).
\]
\end{lem}
\begin{proof}
	Observe
\[
N_{\chi}(r; s)  \leq N_{\chi}(r; 1+it) \leq N_{\chi}(2r; 1+r+it) 
\]
so it suffices to bound the latter quantity. 
Now, if $s_0=1+r+it$, notice
\[
N_{\chi}(2r; s_0)
	\leq 4r\sum_{|1+it - \rho|\leq 2r}\Re\left\{ \frac{1}{s_0-\rho}\right\} 
	\leq 4r\sum_{\rho}\Re\left\{ \frac{1}{s_0-\rho}\right\}.
\]
Applying \cref{ExplicitFormula} and \cref{digamma} twice and noting $\Re\left\{\frac{L'}{L}(s_0, \chi)\right\} \leq  -\frac{\zeta_K'}{\zeta_K}(1+r)$, the above is
\begin{align*}
&\leq 4r \left(\Re\left\{\frac{L'}{L}(s_0, \chi)\right\}+\frac{1}{2}\log D_{\chi}+\Re\left\{\frac{\gamma_{\chi}'}{\gamma_{\chi}}(s_0)\right\}+\delta(\chi)\Re\left\{\frac{1}{s_0}+\frac{1}{s_0-1}\right\}\right)\\
&\leq 4r \left(-\frac{\zeta_K'}{\zeta_K}(1+r) + \frac{1}{2}\log D_{\chi}  + \frac{n_K}{2} \log(|s_0|+1)  +  \delta(\chi)(1+r^{-1}) \right) \\
&\leq 4r \left(\frac{1}{2}\log(D_{K} D_{\chi})   + \frac{n_K}{2} \log(|s_0|+1) + \frac{n_K}{2} +  (1+\delta(\chi))(1+r^{-1}) \right) \\
&\leq \{ 4\log D_K + 2 \log \N \kf_{\chi} +2n_K\log(|t|+3) + 2n_K + 4 + 4\delta(\chi) \} \cdot r + 4 + 4\delta(\chi)
\end{align*}
as $D_{\chi} = D_K \N\kf_{\chi}$. 
\end{proof}

To improve the bound in \cref{ZerosInCircle-Classical}, we exhibit an explicit inequality involving the logarithmic derivative of $L(s,\chi)$ comparable with \cite[Theorem 2]{KadiriNg} for the Dedekind zeta function. 

\begin{prop} \label{ExplicitInequality-Convexity} Let  $0 < \epsilon < \tfrac{1}{4}, T  \geq 1$ and $s = \sigma+it$. For a primitive Hecke character $\chi$, define a multiset of non-trivial zeros of $L(s,\chi)$ by
\[
\cZ_{r, t} = \{ \rho = \beta+i\gamma \, \mid \, L(\rho, \chi) = 0,  |1+it - \rho| < r \}. 
\]
Then, for $0 < r < \epsilon$, 
\begin{equation}
 - \Re\Big\{ \frac{L'}{L}(s,\chi) \Big\}  \leq \big( \tfrac{1}{4} + \tfrac{\epsilon}{\pi} \big)\cL_{\chi} + 4 \epsilon^2 \cL_{\chi}' + \delta(\chi)  \Re\Big\{ \frac{1}{s-1} \Big\}   - \sum_{\rho \in \cZ_{r,t} }  \Re\Big\{\frac{1}{s-\rho} \Big\}  + O_{\epsilon}(n_K)
\label{eqn:EI_1}
\end{equation}
and
\begin{equation}
 - \Re\Big\{ \frac{L'}{L}(s,\chi) \Big\}  \leq \big( \tfrac{1}{4} + \tfrac{\epsilon}{\pi} \big)\cL_{\chi}   + \delta(\chi)  \Re\Big\{ \frac{1}{s-1} \Big\} + O_{\epsilon}(n_K)
\label{eqn:EI_2}
\end{equation}
uniformly in the region
\[
1 < \sigma \leq 1 + \epsilon, \qquad |t| \leq T,
\]
where $\cL_{\chi} = \log D_{\chi} + n_K \log(T+3)$ and $\cL_{\chi}' = \log D_K + \cL_{\chi}$. 
\end{prop}
\begin{proof} This result is a modified version of \cite[Lemma 4.3]{Zaman_2015a} which is motivated by \cite[Lemma 3.1]{HBLinnik}. Consequently, we sketch the argument found in \cite{Zaman_2015a} highlighting the necessary modifications. Assume $\chi$ is non-trivial.  Apply \cite[Lemma 3.2]{HBLinnik} with $f( \, \cdot \,) = L(\, \cdot \, ,\chi), a=s$ and $R = 1-\eta$ where $\eta = \eta_{s,\chi} \in (0,\tfrac{1}{10})$ is chosen sufficiently small so that $L(w,\chi)$ has no zeros on the circle $|w-s| = R$. Then
\begin{equation}
-\Re \Big\{ \frac{L'}{L}(s,\chi) \Big\} = -\sum_{|s-\rho| < R} \Re\Big\{ \frac{1}{s-\rho} - \frac{s-\rho}{R^2} \Big\} - J
\label{JensenUse}
\end{equation}
where
\[
J :=  \int_0^{2\pi} \frac{\cos \theta}{\pi R} \cdot \log| L(s+ R e^{i\theta},\chi)|d\theta.
\]
To lower bound $J$, write
\[
J = \int_0^{\pi/2} + \int_{\pi/2}^{3\pi/2} + \int_{3\pi/2}^{2\pi} = J_1 + J_2 + J_3, 
\]
say, so we may consider each contribution separately. For $J_1$, notice
\[
\log| L(s+ R e^{i\theta},\chi)| \leq \log \zeta_K(\sigma+ R\cos \theta) \ll n_K \log\Big( \frac{1}{\sigma-1 + R\cos \theta}\Big).
\]
Writing $[0,\tfrac{\pi}{2}] = [0,\tfrac{\pi}{2} - (\sigma-1)] \cup [\tfrac{\pi}{2}-(\sigma-1), \tfrac{\pi}{2}] = I_1 \cup I_2$, say. Then 
\[
J_1 = \int_{I_1} + \int_{I_2} \ll n_K \int_{I_1} \cos \theta \log(1/\cos \theta) d\theta + n_K \log(1/(\sigma-1)) \int_{I_2} \cos \theta d\theta \ll_{\epsilon} n_K.
\]
A similar argument holds for $J_3$ so 
\[
J_1 + J_3 \ll_{\epsilon} n_K.
\]
For $J_2$, consider $\theta \in [\pi/2, 3\pi/2]$. As $1 < \sigma \leq 1 + \epsilon$ and $R < 1$,
\[
0 < \sigma + R\cos\theta \leq 1+ \epsilon.
\]
Hence, by \cref{Rademacher}, 
\begin{align*}
 \log|L(s+Re^{i\theta},\chi)| 
 & \leq \tfrac{1}{2} \cL_{\chi} (1-\sigma- R \cos\theta + \epsilon) + O(\epsilon^{-1}n_K)    \\
 & \leq \tfrac{1}{2}  \cL_{\chi}( -R \cos\theta + \epsilon) +  O(\epsilon^{-1}n_K).
\end{align*}
Thus,
\[
J_2 \geq \frac{\cL_{\chi}}{2 \pi R} \int_{\pi/2}^{3\pi/2} -R \cos^2 \theta + \epsilon \cos \theta d\theta + O_{\epsilon}(n_K) 
\]
yielding overall
\begin{equation}
J \geq  - ( \tfrac{1}{4}+ \tfrac{\epsilon}{\pi R} )\cL_{\chi}  + O_{\epsilon}(n_K).
\label{Jintegral}
\end{equation}
For the sum over zeros in \eqref{JensenUse}, observe that the terms are non-negative so \eqref{eqn:EI_2} follows immediately from \eqref{JensenUse} and \eqref{Jintegral} after taking $\eta \rightarrow 0$ which implies $R \rightarrow 1$.  To prove \eqref{eqn:EI_1}, consider $0 < r  <  \tfrac{1}{4}$. By the same observation, we may restrict our sum over zeros from $|s-\rho| < R$ to a smaller circle within it: $|1+it-\rho| < r$. As $r < \epsilon < 1/4$ by assumption, we discard the zeros outside this smaller circle. For such zeros $\rho$ satisfying $|1+it-\rho| < r$, notice
\[
\Re \{ s- \rho  \} = \sigma - \beta < \epsilon + r< 2\epsilon
\]
implying, by \cref{ZerosInCircle-Classical}, that 
\begin{equation}
\sum_{|1+it-\rho| < r} \Re\big\{ \frac{s-\rho}{R^2} \} \leq \frac{2\epsilon}{R^2} \cdot  \big\{ \big(2\cL_{\chi}' + 2n_K + 8\big) r + 8\big\} \leq \frac{4 \epsilon^2}{R^2} \cL_{\chi}' + O(n_K).
\label{SmallerCircle}
\end{equation}
Thus, \eqref{eqn:EI_1} immediately follows upon combining \eqref{JensenUse}, \eqref{Jintegral}, and \eqref{SmallerCircle}, and taking $\eta \rightarrow 0$ which implies $R \rightarrow 1$. This completes the proof for $\chi$ non-trivial. For $\chi = \chi_0$ trivial, we apply the same modifications as described at the end of the proof of \cite[Lemma 4.3]{Zaman_2015a}. 
\end{proof}

\begin{lem} \label{ZerosInCircle-Convexity} Let $\chi \pmod{\kq}$ be given and $0 < r < \epsilon < 1/4$. If $s = \sigma+it$ and
\[
N_{\chi}(r; s) = \#\{ \rho = \beta + i\gamma : L(\rho, \chi) = 0, |s-\rho| \leq r\}
\]
then
\[
N_{\chi}(r; s) \leq \{ 1 +  \tfrac{4}{\pi}\epsilon + 16 \epsilon^2\}   \big( 2 \log D_K  +  \log \N\kf_{\chi} +  n_K \log( |t|+ 3) + O_{\epsilon}(n_K) \big) \cdot r+ 4 + 4\delta(\chi). 
\]
\end{lem}
\begin{proof} The proof is analogous to \cref{ZerosInCircle-Classical} using \cref{ExplicitInequality-Convexity} in place of \cref{ExplicitFormula,digamma}.   
\end{proof}

\subsection{Arithmetic Sums} 

We estimate various sums over integral ideals of $K$ which requires some additional notation.  Recall that the Dedekind zeta function $\zeta_K(s)$ is the primitive Hecke $L$-function, defined by \eqref{def:Hecke_L-fcn}, associated to the trivial character $\chi_0$. Namely, 
\[
\zeta_K(s) = \sum_{\kn \subseteq\cO_K} (\N\kn)^{-s} = \prod_{\kp}\Big(1-\frac{1}{\N\kp^s} \Big)^{-1}
\]
for $\Re\{s\} > 1$. Since $\zeta_K(s)$ has a simple pole at $s=1$, we may define
\begin{equation}
\kappa_K := \Res_{s=1} \zeta_K(s) \quad \text{ and } \quad \gamma_K := \kappa_K^{-1} \lim_{s \rightarrow 1} \Big( \zeta_K(s) - \frac{\kappa_K}{s-1} \Big)
\label{def:Zeta_Laurent}
\end{equation}
so the Laurent expansion of $\zeta_K(s)$ at $s=1$ is given by 
\[
\zeta_K(s) =  \frac{\kappa_K}{s-1} + \kappa_K \gamma_K + O_K( |s-1| ).
\]
We refer to $\gamma_K$ as the \emph{Euler-Kronecker constant of $K$}, which was first introduced by Ihara \cite{Ihara}. For further details on $\gamma_K$, see \cite{Ihara,Ihara2,Murty} for example.
\begin{lem}
\label{lem:harmonic_sum}
For $x > 0$ and $\eta > 0$,
\[
\left|\sum_{\mathrm{N}\kn\leq x}\frac{1}{\mathrm{N}\kn}\left(1-\frac{\mathrm{N}\kn}{x}\right)^{n_K}-\kappa_K\left(\log x-\sum_{j=1}^{n_K} \frac{1}{j}\right)-\kappa_K\gamma_K\right|\ll e^{O_{\eta}(n_K)} \big( n_K^{n_K} D_K \big)^{\tfrac{1}{4}+\eta} x^{-\frac{1}{2}}.
\]
\end{lem}
\begin{proof}
Without loss, we may assume $\eta \in (0,1/2)$. Observe
{\small
\begin{align*}
\sum_{\mathrm{N}\kn\leq x}\frac{1}{\mathrm{N}\kn}\left(1-\frac{\mathrm{N}\kn}{x}\right)^{n_K}-\kappa_K\left(\log x-\sum_{j=1}^{n_K} \frac{1}{j}\right)-\kappa_K\gamma_K=\frac{1}{2\pi i}\int_{-\frac{1}{2}-i\infty}^{-\frac{1}{2}+i\infty}\zeta_K(s+1)\frac{x^s}{s}\frac{n_K!}{\prod_{j=1}^{n_K} (s+j)}ds.
\end{align*}}%
Using Lemma \ref{Rademacher} and noting $\zeta_{\Q}(1+\eta)^{n_K} \ll e^{O_{\eta}(n_K)}$, it follows that
\begin{align*}
&\frac{1}{2\pi i}\int_{-\frac{1}{2}-i\infty}^{-\frac{1}{2}+i\infty}\zeta_K(s+1)\frac{x^s}{s}\frac{n_K!}{\prod_{j=1}^{n_K} (s+j)}ds\\
&\ll e^{O_{\eta}(n_K)} D_{K}^{\frac{1}{4} +\eta}x^{-1/2}n_K!\int_{-\infty}^{\infty}(1+|t|)^{(\frac{1}{4}+\eta)n_K}\left|\frac{\Gamma(-\frac{1}{2}+it)}{\Gamma(\frac{1}{2}+n_K+it)}\right|dt\\
&\ll e^{O_{\eta}(n_K)}D_{K}^{\frac{1}{4}+\eta}x^{-1/2}n_K!\int_{-n_K}^{n_K}(1+|t|)^{(\frac{1}{4}+\eta)n_K}\left|\frac{\Gamma(-\frac{1}{2}+it)}{\Gamma(\frac{1}{2}+n_K+it)}\right|dt\\
&\ll \frac{e^{O_{\eta}	(n_K)}  D_{K}^{\frac{1}{4}+\eta}x^{-1/2}n_K!}{\Gamma(n_K+\frac{1}{2})} (n_K^{n_K})^{\frac{1}{4}+\eta} \\
&\ll e^{O_{\eta}(n_K)} (n_K^{n_K}D_{K})^{\frac{1}{4}+\eta}x^{-1/2} 
\end{align*} 
as claimed.
\end{proof}

\begin{cor}
\label{cor:harmonic_sum}
Let $\eta > 0$ and $C_1 = C_1(\eta) \geq 3$ be sufficiently large. If
\[
x \geq C_1 e^{O_{\eta}(n_K)}  \big( n_K^{n_K} D_K)^{1/2+\eta},
\]
then
\[
\sum_{\mathrm{N}\kn \leq x} \frac{1}{\mathrm{N}\kn} \gg_{\eta} \kappa_K \log x.
\]
\end{cor}
\begin{proof} If $\kappa_K \leq 1/\log x$ then the claim follows from the trivial bound $\sum_{\mathrm{N}\kn \leq x} \frac{1}{\mathrm{N}\kn} \geq 1$. Otherwise, we may assume $\kappa_K \geq 1/\log x$. From \cref{lem:harmonic_sum}, it follows
\[
\frac{1}{\kappa_K} \sum_{\mathrm{N}\kn \leq x} \frac{1}{\mathrm{N}\kn} \geq  \log x - \sum_{j=1}^{n_K} \frac{1}{j} + \gamma_K + O\left( \frac{ e^{O_{\eta}(n_K)} (n_K^{n_K} D_K)^{1/4+\eta}  \log x}{\sqrt{x}} \right).
\]
By \cite[Proposition 3]{Ihara} ,
\[
\gamma_K \geq - \frac{1}{2} \log D_K  + \frac{\gamma_{\Q} + \log 2\pi}{2} \cdot n_K - 1
\]
where $\gamma_{\Q} = 0.577\dots$ is the classical Euler's constant. 
Bounding $\sum_{1 \leq j \leq n_K} j^{-1} \leq \log n_K + 1$ and using the condition on $x$, we deduce from the previous inequality that 
\begin{align*}
\frac{1}{\kappa_K} \sum_{\mathrm{N}\kn \leq x} \frac{1}{\mathrm{N}\kn} 
& \geq (\log x) \{ 1 + O( C_1^{-1/4} ) \} - \frac{1}{2}\log D_K + \frac{\gamma_{\Q} + \log 2\pi}{2} \cdot n_K - \log n_K - 2  \\
& \geq (\log x) \{ 1 + O( C_1^{-1/4} ) \} - \frac{1}{2}\log D_K  - 1 \\
& \geq (\log x) \{ 1 + O( C_1^{-1/4} + (\log x)^{-1} ) \} - \frac{1}{1+2\eta}\log x  \\
& \geq (\log x) \{ \eta + O(C_1^{-1/4} + (\log x)^{-1} ) \}. 
\end{align*}
Since $x \geq C_1 = C_1(\eta)$ and $C_1(\eta)$ is sufficiently large, the desired bound follows. 
\end{proof}

Taking the logarithmic derivative of $\zeta_K(s)$ yields in the usual way
\begin{equation}
-\frac{\zeta_K'}{\zeta_K}(s) = \sum_{\kn \subseteq \cO_K} \frac{\Lambda_K(\kn)}{(\N\kn)^s}
\label{eqn:LogDiffZeta}
\end{equation}
for $\Re\{s\}>1$, where $\Lambda_K( \, \cdot \,)$ is the von Mangoldt $\Lambda$-function of the field $K$ defined by
\begin{equation}
\Lambda_K(\kn) = 
	\begin{cases}
		\log \N\kp & \text{if $\kn$ is a power of a prime ideal $\kp$,} \\
		0 & \text{otherwise.}
	\end{cases}
	\label{def:vonMangoldt}
\end{equation}
Using this identity, we prove a simple elementary lemma.

\begin{lem} \label{lem:PrimeSum}
	For $y \geq 2$,
	\[
	\sum_{\N\kn \leq y} \frac{\Lambda_K(\kn)}{\N\kn} \ll \log(D_Ky).
	\]	
\end{lem}
\begin{proof}
	Denote $\sigma = 1+ \frac{1}{\log y}$. From \eqref{eqn:LogDiffZeta}, it follows
	\[
	\sum_{\N\kn \leq y} \frac{\Lambda_K(\kn)}{\N\kn} \leq e	\sum_{\kn} \frac{\Lambda_K(\kn)}{\N\kn^\sigma} =  -\frac{\zeta_K'}{\zeta_K}(\sigma). 
	\]
	By \cref{ExplicitFormula,digamma}, the RHS is
	\[
	\leq \frac{1}{2}	\log D_K + \log y + 1 - \sum_{\rho} \Re\Big\{ \frac{1}{\sigma-\rho} \Big\} + O(n_K).
	\]
	As $\Re\{ (\sigma-\rho)^{-1}\} \geq 0$ and $n_K \ll \log D_K$, the claim follows.
\end{proof}
Finally, we end this section with a bound for $h_H$ in terms of $n_K, D_K,$ and $Q=Q_H$.

\begin{lem}
\label{lem:h_H-Bound}
	Let $H$ be a congruence class group of $K$. For $\epsilon > 0$, 
	\[
	h_H \leq e^{O_{\epsilon}(n_K)} D_K^{1/2+\epsilon} Q^{1+\epsilon}.
	\] 
\end{lem}
\begin{proof} Observe, by the definitions of $Q$ and $\mathfrak{f}_H$ in \cref{subsec:notation}, that if $\chi$ is a Hecke character satisfying $\chi(H) = 1$ then $\kf_{\chi} \mid \kf_H$ and $\N\kf_{\chi} \leq Q$. Hence,  
	\[
	h_H = \sum_{ \substack{\chi \pmod{\kf_H} \\ \chi(H) =1} } 1   \leq \sum_{\substack{ \N\kf \leq Q \\ \kf \,\mid\, \kf_H} } \sum_{\chi \pmod{\kf} } 1 = \sum_{\substack{ \N\kf \leq Q \\ \kf \,\mid\, \kf_H} } \#\mathrm{Cl}(\kf). 
	\]	
	Recall the classical bound $\#\mathrm{Cl}(\kf) \leq 2^{n_K} h_K \N\kf$ where $h_K$ is the class number of $K$ (in the broad sense) from \cite[Theorem 1.7]{milneCFT}, for example. Bounding the class number using Minkowski's bound (see \cite[Lemma 1.12]{Weiss} for example), we deduce that
	\[
	h_H \leq \sum_{\substack{ \N\kf \leq Q \\ \kf \,\mid\, \kf_H} } e^{O_{\epsilon}(n_K)} D_K^{1/2+\epsilon} \N\kf \leq e^{O_{\epsilon}(n_K)} D_K^{1/2+\epsilon} Q^{1+\epsilon}  \sum_{ \kf \,\mid\, \kf_H} \frac{1}{(\N\kf)^{\epsilon}}.
	\]
	For the remaining sum, notice 
	\[
	 \sum_{ \kf \,\mid\, \kf_H} \frac{1}{(\N\kf)^{\epsilon}} \leq \prod_{\kp \mid \kf_H} \Big(1-\frac{1}{\N\kp^{\epsilon}} \Big)^{-1} \leq \exp\Big( O\Big( \sum_{\kp \mid \kf_H} \frac{1}{\N\kp^{\epsilon}} \Big) \Big)\leq e^{O(\omega(\kf_H))},
	\]
	where $\omega(\kf_H)$ is the number of prime ideals $\kp$ dividing $\kf_H$. From \cite[Lemma 1.13]{Weiss}, we have $\omega(\kf_H) \ll  O_{\epsilon}(n_K)  + \epsilon \log(D_KQ)$ whence the desired estimate follows after rescaling $\epsilon$. 
\end{proof}

\begin{remark}
Weiss \cite[Lemma 1.16]{Weiss} achieves a comparable bound with $Q^{1+\epsilon}$ replaced by $\mathrm{N}\kf_H$.  This seemingly minor difference will in fact play a key role in improving the range of $T$ in \cref{LFZD-MainTheorem}.
\end{remark}

\section{Proof of \cref{LFZD-MainTheorem}}
\label{sec:LFZD_Proof}

In this section, we deduce \cref{LFZD-MainTheorem} from two key results. Without loss, we may assume $H \pmod{\kq}$ is a primitive congruence class group of $K$. Recall
\[
n_K = [K:\Q], \qquad D_K = |\mathrm{disc}(K/\Q)|
\]
and
\[
h_H = [I(\kq):H], \qquad Q = Q_H = \max\{ \N\kf_{\chi} : \chi(H) = 1\}.
\]
First, we require a mean value theorem for certain Dirichlet polynomials.

\begin{thm} \label{thm:LargeSieve}
Let $\epsilon > 0$ be arbitrary.  Let $b$ be a complex-valued function on the prime ideals $\kp$ of $K$ such that
\[
\sum_{\kp}|b(\kp)|<\infty
\]
and $b(\kp)=0$ when $\N\kp\leq y$.  Let $H \pmod{\kq}$ be a primitive congruence class group of $K$.  If $T\geq1$ and
\begin{equation}
y\geq \big\{ h_H n_K^{\tfrac{5n_K}{4}} D_K^{\tfrac{3}{2}} Q^{\tfrac{1}{2}} T^{\tfrac{n_K}{2}+1}  \big\}^{1+\epsilon} e^{O_{\epsilon}(n_K)}
\label{eqn:LargeSieve_yRange}
\end{equation}
then
\[
\sum_{\substack{\chi \pmod{\kq} \\ \chi(H)=1}}\int_{-T}^{T}\left|\sum_{\kp}b(\kp)\chi(\kp)\N\kp^{-it}\right|^2 dt\ll_{\epsilon} \frac{1}{\log y}\sum_{\kp}\N\kp|b(\kp)|^2.
\]
\end{thm} 
\begin{remark}
Weiss proved essentially the same result \cite[Corollary 3.8]{Weiss} as \cref{thm:LargeSieve} but with condition \eqref{eqn:LargeSieve_yRange} replaced by
\[
y \geq (h_H n_K^{2 n_K} D_K Q T^{2n_K})^8. 
\] 
The exponent 8 happens to be large enough so that it inflates $\Cr{c3}$ and $\Cr{c4}$ in \eqref{eqn:Weiss}.  The purpose of \cref{thm:LargeSieve} is to ensure that the size of $y$ in the required large sieve inequality does not affect the exponents in \cref{LFZD-MainTheorem}.  Our proof mostly follows Weiss' arguments but with more careful analysis.
\end{remark}

The second ingredient is a method for detecting zeros of Hecke $L$-functions. To simplify its statement, define
\begin{equation}
\cL := 2 \log D_K + \log Q + n_K \log(T+3) +  \Theta n_K,
\label{def:cL}
\end{equation}
where $\Theta \geq 1$ is sufficiently large, and let $\mathbf{1}( \cdot )$ be an indicator function.  

\begin{prop}
\label{prop:ZeroDetector}
Let $\chi$ be a Hecke character satisfying $\chi(H) = 1$ and $\epsilon \in (0,1/4)$ be arbitrary. Suppose $L(s,\chi)$ has a non-trivial zero $\rho$ satisfying
\[
|1+i\tau - \rho| \leq r
\]
for
\[
\frac{R}{\cL} < r < r_0 \qquad |\tau| \leq T
\]
where $T \geq 1$ is arbitrary, $R \geq 1$ is sufficiently large, and $0 < r_0 < \frac{\epsilon}{3.8}$. Then 
\[
e^{- 73.2 \phi r \cL}   \ll  r^4 \cL  \int_y^x \Big| \sum_{y \leq \N\kp < u} \frac{\chi(\kp) \log \N \kp  }{\N\kp^{1+i\tau}}  \Big|^2 \frac{du}{u}  + \delta(\chi) \mathbf{1}_{\{ |\tau| < 4r \}}(\tau)
\]
where $\phi = 1 + \tfrac{4}{\pi} \epsilon + 16\epsilon^2$ and provided $x,y \geq 1$ satisfy
\begin{equation}
\cL \leq \log y \leq 2.3 \phi \cL \quad \text{ and} \quad 122 \phi \cL \leq \log x \ll \cL.
\end{equation}
\end{prop}
Weiss \cite[Lemma 4.2]{Weiss} showed a similar estimate but without any explicit constants. As such, the proof of \cref{prop:ZeroDetector}, which is contained in \cref{sec:ZeroDetector}, follows his overall arguments using Tur\'{a}n power sums but with a more careful numerical analysis.   

Combining these two components allows us to establish \cref{LFZD-MainTheorem}. 

\subsubsection*{Proof of \cref{LFZD-MainTheorem} from \cref{thm:LargeSieve} and \cref{prop:ZeroDetector}:}  Without loss, we may assume $H \pmod{\kq}$ is primitive because $Q= Q_H = Q_{H^*}$ and $h_H = h_{H^*}$ if $H^*$ induces $H$. If $n_K=1$ then the desired bound follows from the combined works of Huxley \cite{Huxley} and Jutila \cite{Jutila}.  Hence, we may also assume $n_K \geq 2$.

First, suppose 
\[
\frac{1}{2} \leq \sigma \leq 1 -  \frac{0.05}{4}.
\]
By a naive application of \cite[Lemma 2.1]{LMO}, one can directly verify that for $T \geq 1$,
\begin{equation}
\sum_{\chi(H)=1} N(\sigma, T,\chi) \ll h_H T \log( D_K Q T^{n_K}) \ll ( e^{O(n_K)} D_K^2QT^{n_K})^{81(1-\sigma)} 
\label{eqn:ThickStrip}
\end{equation}
after bounding $h_H$ with \cref{lem:h_H-Bound}.

Now, let $\epsilon \in (0,1/4)$ be fixed and denote $\phi := 1 + \tfrac{4}{\pi} \epsilon + 16\epsilon^2$. Suppose
\begin{equation}
\label{eqn:SigmaCondition}
1- \frac{\epsilon}{4} \leq \sigma < 1. 
\end{equation}
Let $R \geq 1$ be fixed and sufficiently large. By applying the bound in \cref{lem:h_H-Bound} to \cite[Theorem 4.3]{Weiss}, we deduce that for $T \geq 1$,
\begin{equation}
\sum_{\chi(H)=1} N(1-\tfrac{R}{\cL}, T,\chi) \ll 1,
\label{eqn:ThinStrip}
\end{equation}
so it suffices to bound the number of zeros $\rho = \beta+i\gamma$ satisfying
\begin{equation}
\sigma < \beta < 1 - \frac{R}{\cL} \qquad |\gamma| \leq T.
\label{eqn:ZeroRegion}
\end{equation}

Fix $\eta \in (0,1)$ sufficiently small and let $r = (1+\eta)(1-\sigma)$ so by \eqref{eqn:SigmaCondition}, we have $r < \tfrac{\epsilon}{3.8}$.  For each zero $\rho = \beta+i\gamma$ of $L(s,\chi)$ satisfying \eqref{eqn:ZeroRegion}, define
\[
\Phi_{\rho, \chi}(\tau) := \mathbf{1}_{\{ |1+i\tau-\rho| \leq r\}}(\tau)
\]
so by assumption
\[
r^{-1} \int_{-T}^T \Phi_{\rho, \chi}(\tau) d\tau \gg 1.
\]
Select $y = e^{2.3\phi \cL}$ and $x = e^{122\phi \cL}$. By \cref{prop:ZeroDetector}, it follows that
\[
e^{-73.2 \phi r \cL} \ll \int_{-T}^T r^{-1} \Phi_{\rho,\chi}(\tau)  \Big( r^4 \cL  \int_y^x \Big| \sum_{y \leq \N\kp < u} \frac{\chi(\kp) \log \N \kp  }{\N\kp^{1+i\tau}}  \Big|^2  \frac{du}{u}  +  \delta(\chi) \mathbf{1}_{\{|\tau| < 4r\}}(\tau) \Big) d\tau.
\]
Summing over all zeros $\rho$ of $L(s,\chi)$ satisfying \eqref{eqn:ZeroRegion} and using  \eqref{eqn:ThinStrip}, we have that
\begin{equation}
	\label{eqn:ChiZeros}
	\begin{aligned}
	e^{-73.2\phi r\cL} \cdot N(\sigma, T,\chi) 
	& \ll r^4 \cL^2  \int_y^x \Big( \int_{-T}^T  \Big| \sum_{y \leq \N\kp < u} \frac{\chi(\kp) \log \N \kp  }{\N\kp^{1+i\tau}}  \Big|^2 d\tau  \Big) \frac{du}{u}   \\
	& \qquad\qquad +  \delta(\chi) \cL \int_{-T}^T \mathbf{1}_{\{|\tau| < 4r\}}(\tau)   d\tau	+ 1
	\end{aligned}
\end{equation}
since, for $|\tau| \leq T$, 
\[
\sum_{\substack{ \rho \\ L(\rho, \chi) = 0} } \Phi_{\rho,\chi}(\tau) \ll N_{\chi}(r; 1+i\tau) \ll r\cL
\]
by \cref{ZerosInCircle-Classical}. Summing \eqref{eqn:ChiZeros} over $\chi$ satisfying $\chi(H) = 1$, we obtain
\begin{equation}
	\begin{aligned}
	e^{-73.2 \phi r\cL} \cdot \sum_{\chi(H) = 1} N(\sigma, T,\chi) 
	& \ll r^4 \cL^2  \int_y^x \Big(\sum_{\chi(H)=1} \int_{-T}^T  \Big| \sum_{y \leq \N\kp < u} \frac{\chi(\kp) \log \N \kp  }{\N\kp^{1+i\tau}}  \Big|^2 d\tau  \Big) \frac{du}{u}  + r\cL\\
	\end{aligned}
	\label{LFZD-PreLargeSieve}
\end{equation}
Observe that, for $\nu = \nu(\epsilon) > 0$ fixed and sufficiently small, \cref{lem:h_H-Bound} implies
\[
y = e^{2.3\phi \cL} \geq D_K^{4.6\phi} Q^{2.3\phi} T^{2.3 \phi n_K} e^{2.3\phi \Theta n_K} \geq \{ h_H n_K^{1.25 n_K} D_K^{1.5} Q^{0.5} T^{0.5 n_K + 1} \} ^{1+\nu} e^{\Theta n_K}
\]
since $T \geq \max\{ n_K D_K^{-2/n_K} Q^{-3/5n_K}, 1\}$, $n_K \geq 2$, and $\Theta \geq 1$ is sufficiently large. Thus $y$ satisfies the conditions of \cref{thm:LargeSieve}, so the RHS of \eqref{LFZD-PreLargeSieve} is
\begin{equation}
\ll r^4 \cL^2 \int_y^x \frac{1}{\log y} \sum_{y \leq \N\kp < u} \frac{ (\log \N\kp)^2}{\N\kp} \frac{du}{u} + r\cL.
\label{LFZD-Penultimate}
\end{equation}
For the sum over prime ideals, note by \cref{lem:PrimeSum} 
\[
\sum_{y \leq \N\kp < u} \frac{ (\log \N\kp)^2}{\N\kp} \ll (\log u)^2
\]
since $u \geq y  = e^{2.3 \phi\cL} \geq 2D_K$.  Hence, the previous expression is 
\begin{align*}
& \ll r^4 \cL^2 \int_y^x \frac{(\log u)^2}{u \log y}du + r\cL \\
& \ll r^4 \cL^2 \frac{ (\log x)^3 }{\log y} + r\cL \\
& \ll r^4 \cL^4 
\end{align*}
as $\log y \asymp \log x \asymp \cL$. Comparing with \eqref{LFZD-PreLargeSieve} and \eqref{LFZD-Penultimate}, we have shown 
\[
\sum_{\chi(H) = 1} N(\sigma, T,\chi) \ll (r\cL)^4 e^{73.2 \phi r\cL} \ll e^{(73.2 \phi +\eta) (1+\eta) (1-\sigma) \cL} 
\]
as $r = (1+\eta)(1-\sigma)$ and $(r\cL)^4 \ll e^{\eta r\cL}$. In light of \eqref{eqn:ThickStrip} and \eqref{eqn:SigmaCondition}, both cases follow from the respective choices $\epsilon = 0.05$ and $\epsilon = 0.001$ and recalling $\eta$ is fixed and sufficiently small. \qed
\section{Mean Value of Dirichlet Polynomials}
\label{sec:MeanValue}

In \cite{Gallagher}, Gallagher proves a large sieve inequality of the following form.
\begin{thm}
\label{thm:Gallagher}
Let $\{a_n\}$ be a sequence of complex numbers such that $\sum_{n\geq 1}n|a_n|^2<\infty$.  Assume that $a_n=0$ if $n$ has any prime factor less than $R\geq2$.  If $T\geq 1$, then
\[
\sum_{q\leq R}\log\frac{R}{q}~\sideset{}{^*}\sum_{\chi\bmod q}\int_{-T}^{T}\Big|\sum_{n\geq 1}a_n\chi(n)n^{it}\Big|^2 dt\ll \sum_{n\geq 1}(R^2 T+n)|a_n|^2,
\]
where $\sideset{}{^*}\sum$ denotes a restriction of the summation to primitive Dirichlet characters.
\end{thm} 

The $\log R/q$ savings, which arises from forcing $a_n=0$ when $n$ has a small prime factor, turns out to be decisive in certain applications, such as Bombieri's proof of \eqref{eqn:Linnik} in \cite{Bombieri}.  The key ingredients in proof of \cref{thm:Gallagher} are the duality argument, properties of Gauss sums, and the fact that the Farey fractions up to height $R$ are $R^{-2}$-well-spaced (cf. \cite[Sections 7.3-7.4]{IK}); apart from the duality argument, sufficiently strong analogues of these results over number fields for the purpose of replacing the Dirichlet characters in \cref{thm:Gallagher} with Hecke characters do not exist yet.  In order to circumvent these deficiencies, we use the Selberg sieve to prove a variant of \cref{thm:Gallagher} where the $\log R/q$ term on the left hand side is translated to a $(\log R)^{-1}$ savings on the right hand side.  The use of the Selberg sieve introduces several sums over integral ideals whose evaluation requires smoothing.  Ultimately, this introduces the factor of $n_K$ the lower bound for $T$ in \cref{LFZD-MainTheorem}.

\subsection{Preparing for the Selberg Sieve} To apply the Selberg sieve, we will require several weighted estimates involving Hecke characters. Before we begin, we highlight the necessary properties of our weight $\Psi$. 
 
\begin{lem} \label{lem:WeightFunction}
	For $T \geq 1$, let $A = T \sqrt{2n_K}$. There exists a weight function $\Psi(x) \in C_c((0,\infty))$ with Mellin transform $\widehat{\Psi}(s)$ such that:
	\begin{enumerate}[(i)]
		\item $0 \leq \Psi(x) \leq A/2$ and $\Psi(x)$ vanishes outside the interval
				\[ 
				 e^{-2n_K/A} \leq x \leq e^{2n_K/A}
				\]
		\item $\widehat{\Psi}(s)$ is an entire function and further $\widehat{\Psi}(s) = \Big[ \frac{\sinh(s/A)}{s/A} \Big]^{2n_K}$. 
		\item For all complex $s = \sigma+it$, 
			\[
				|\widehat{\Psi}(s)| \leq \Big( \frac{A}{|s|} \Big)^{2n_K} e^{|\sigma|/A}.
			\]
		\item For $|s| \leq A$,
			\[
			|\widehat{\Psi}(s)| \leq \Big(1 + \frac{|s|^2}{5A^2}\Big)^{2n_K}. 
			\]
		\item Uniformly for $|\sigma| \leq A/\sqrt{2n_K}$, 
			\[
			    |\widehat{\Psi}(s)| \ll 1.
			\]
		\item Let $\{b_m\}_{m \geq 1}$ be a sequence of complex numbers with $\sum_m |b_m| < \infty$. Then
			\[
			\int_{-T}^T \Big| \sum_m b_m m^{-it} \Big|^2 dt \ll \int_0^{\infty} \Big| \sum_m b_m \Psi\big( \frac{x}{m} \big) \Big|^2 \frac{dx}{x}
			\]
	\end{enumerate}
\end{lem}
\begin{proof}
	See \cite[Lemma 3.2 and Corollary 3.3]{Weiss}; in his notation, $\Psi(x) = H_{2n_K}(x)$ with parameter $A = T\sqrt{2n_K}$. 
\end{proof}
For the remainder of this section, assume:
  \begin{itemize}
  	\item $H \pmod{\kq}$ is an arbitrary \emph{primitive} congruence class group of $K$.
 	\vspace{1mm}
  	\item $0 < \epsilon < 1/2$ and $T \geq 1$ is arbitrary. 
  	\vspace{1mm}
  	\item $\Psi$ is the weight function of \cref{lem:WeightFunction}. 
  \end{itemize}
Next, we establish improved analogues of \cite[Lemmas 3.4 and 3.6 and Corollary 3.5]{Weiss}.

\begin{lem}
\label{lem:smoothed}  Let $\chi \pmod{\kq}$ be a Hecke character satisfying $\chi(H) = 1$.   For $x > 0$,
\[
\left|\sum_{\kn}\frac{\chi(\kn)}{\N\kn}\cdot \Psi\Big(\frac{x}{\N\kn}\Big)-\delta(\chi)\frac{\varphi(\kq)}{\N\kq}\kappa_K\right|
\leq e^{O_{{\epsilon}}(n_K)} \cdot \big\{ n_K^{\frac{n_K}{4}} D_K^{\tfrac{1}{2}} Q^{\tfrac{1}{2}} T^{\tfrac{n_K}{2}+1} \big\}^{1+\epsilon}
\]
\end{lem}

\begin{proof}
We have
\[
\sum_{\kn}\frac{\chi(\kn)}{\N\kn} \cdot \Psi\Big(\frac{x}{\N\kn}\Big)-\delta(\chi)\frac{\varphi(\kq)}{\N\kq}\kappa_K=\frac{1}{2\pi i}\int_{-1-i\infty}^{-1+i\infty}L(s+1,\chi)\widehat{\Psi}(s)x^s ds.
\]
If $\chi\pmod{\kq}$ is induced by the primitive character $\chi^* \pmod{\mathfrak{f}_{\chi}}$, then
\[
L(s,\chi)=L(s,\chi^*)\prod_{\substack{\kp\mid \kq \\ \kp\nmid\mathfrak{f}_{\chi}}}(1-\chi^*(\kp)\N\kp^{-s}) 
\]
implying 
\[
|L(it,\chi)|\leq 2^{\omega(\kq)}|L(it,\chi^*)|
\]
where $\omega(\kq)$ is the number of distinct prime ideal divisors of $\kq$. Since $H \pmod{\kq}$ is primitive,  
\[
\omega(\kq)\leq 6e^{4/\epsilon}n_K+\tfrac{\epsilon}{2} \log(D_K Q),
\]
by \cite[Lemma 1.13]{Weiss}. Hence, for $\Re\{s\} = -1$,
\[
|L(s+1,\chi)| \ll e^{O_{\epsilon}(n_K)} (D_K Q)^{\epsilon/2} |L(s+1,\chi^*)|
\]
Thus, by \cref{Rademacher}, we have 
\begin{align*}
&\left|\frac{1}{2\pi i}\int_{-1-i\infty}^{-1+i\infty}L(s+1, \chi )\widehat{\Psi}(s)x^s ds\right|\\
&\ll e^{O_{\epsilon}(n_K)} (D_K Q)^{\tfrac{1}{2}+\epsilon} x^{-1}\int_{0}^{\infty}(1+|t|)^{(\frac{1}{2}+\epsilon)n_K}|\widehat{\Psi}(-1+it)|dt
\end{align*}
as $D_{\chi} \leq D_K Q$. By \cref{lem:WeightFunction}(iii) and (iv), it follows that
\begin{align*}
&\int_{0}^{\infty}(1+|t|)^{(\frac{1}{2}+\epsilon)n_K}|\widehat{\Psi}(-1+it)|dt\\
&=\int_{0}^{\frac{A}{2}}(1+|t|)^{(\frac{1}{2}+\epsilon)n_K}|\widehat{\Psi}(-1+it)|dt+\int_{\frac{A}{2}}^{\infty}(1+|t|)^{(\frac{1}{2}+\epsilon)n_K}|\widehat{\Psi}(-1+it)|dt\\
&\ll e^{O(n_K)}A^{(\frac{1}{2}+\epsilon)n_K+1}.
\end{align*}
Collecting the above estimates, the claimed bound follows upon recalling $A = T\sqrt{2n_K}$. 
\end{proof}

\begin{cor}
Let $C$ be a coset of the primitive congruence class group $H\pmod{\kq}$, and let $\kd$ be an integral ideal coprime to $\kq$. For all $x > 0$, we have
\[
\left|\sum_{\substack{\kn\in C \\ \kd \mid\kn}}\frac{1}{\N\kn}\Psi\Big(\frac{x}{\N\kn}\Big)-\frac{\varphi(\kq)}{\N\kq}\frac{\kappa_K}{h_H} \cdot \frac{1}{\N\kd}\right|\leq e^{O_{{\epsilon}}(n_K)} \cdot \big\{ n_K^{\frac{n_K}{4}} D_K^{\tfrac{1}{2}} Q^{\tfrac{1}{2}} T^{\tfrac{n_K}{2}+1} \big\}^{1+\epsilon} \cdot \frac{1}{x}.
\]
\end{cor}
\begin{proof}
The proof is essentially the same as that of \cite[Corollary 3.5]{Weiss}, except for the fact that we have an improved bound in Lemma \ref{lem:smoothed}.
\end{proof}

We now apply the Selberg sieve.  For $z\geq 1$, define
\begin{equation}
S_z=\{\kn:\kp\mid\kn\implies\N\kp>z\} \qquad \text{and} \qquad V(z)=\sum_{\N\kn\leq z}\frac{1}{\N\kn}.	
\end{equation}

\begin{lem}
\label{lem:3.3}
Let $C$ be a coset of the primitive congruence class group $H\pmod{\kq}$.  For $x > 0$ and $z \geq 1$,  
\[
\sum_{\kn\in C\cap S_z}\frac{1}{\N\kn} \Psi\Big(\frac{x}{\N\kn}\Big)\leq \frac{\kappa_K}{h_H V(z)}+O\Big( \frac{M z^{2+2\epsilon}}{x} \Big),
\]
where
\begin{equation}
M=e^{O_{{\epsilon}}(n_K)} \cdot \big\{ n_K^{\frac{n_K}{4}} D_K^{\tfrac{1}{2}} Q^{\tfrac{1}{2}} T^{\tfrac{n_K}{2}+1} \big\}^{1+\epsilon}.
\label{def:M}
\end{equation}
\end{lem}
\begin{proof}
The proof is essentially the same as that of \cite[Lemma 3.6]{Weiss}, except for the fact that we have an improved bound in Lemma \ref{lem:smoothed}.
\end{proof}
\subsection{Proof of \cref{thm:LargeSieve}}
Let $z$ be a parameter satisfying $1 \leq z \leq y$, which we will specify later. Applying \cref{lem:WeightFunction} and writing
\[
b_m=\sum_{\N\kn=m}b(\kn)\chi(\kn),
\]
for each Hecke character $\chi$ satisfying $\chi(H) = 1$, it follows that
\[
\sum_{\chi(H)=1}\int_{-T}^{T}\left|\sum_{\kn}b(\kn)\chi(\kn)\N\kn^{-it}\right|^2 dt\ll\int_0^{\infty}\sum_{\chi(H)=1}\left|\sum_{\kn}b(\kn)\chi(\kn)\Psi\Big(\frac{x}{\N\kn}\Big)\right|^2\frac{dx}{x}.
\]
 By the orthogonality of characters and the Cauchy-Schwarz inequality,
\[
 \sum_{\chi(H)=1}\left|\sum_{\kn}b(\kn)\chi(\kn) \Psi\Big(\frac{x}{\N\kn}\Big)\right|^2 \leq
  h_H \sum_{C \in I(\kq)/H} \Big( \sum_{\kn\in C}\N\kn |b(\kn)|^2 \Psi \left(\frac{x}{\N\kn}\right) \Big) \Big(\sum_{\kn\in C\cap S_z}\frac{1}{\N\kn} \Psi\Big(\frac{x}{\N\kn}\Big) \Big)
\]
since $z \leq y$ and $b(\kn)$ is supported on prime ideals with norm greater than $y$. 
By Lemma \ref{lem:3.3}, the RHS is
\begin{align*}
& \leq  \sum_{C \in I(\kq)/H} \sum_{\kn\in C}\N\kn |b(\kn)|^2 \Psi \left(\frac{x}{\N\kn}\right)\Big( \frac{\kappa_K}{V(z)}+ \frac{h_H M'}{x} \Big)  \\
& \leq \sum_{\kn} \N\kn|b(\kn)|^2 \Psi\Big(\frac{x}{\N\kn}\Big) \Big( \frac{\kappa_K}{V(z)}+ \frac{h_H M'}{x} \Big),
\end{align*}
where $M' = M z^{2+2\epsilon}$. 
Combining the above estimates yields
\begin{align*}
&\sum_{\chi(H)=1}\int_{-T}^{T}\left|\sum_{\kp}b(\kp)\chi(\kp)\N\kp^{-it}\right|^2 dt\\
\ll~&\sum_{\kn} \N\kn |b(\kn)|^2\left(\frac{\kappa_K}{V(z)}\int_{0}^{\infty} \Psi\Big(\frac{x}{\N\kn}\Big)\frac{dx}{x}+h_H M'\int_{0}^{\infty}\frac{1}{x}\Psi\Big(\frac{x}{\N\kn}\Big)\frac{dx}{x}\right)\\
\ll~&\sum_{\kn} \N\kn |b(\kn)|^2\left(\frac{\kappa_K}{V(z)}|\widehat{\Psi}(0)|+\frac{h_H M'}{\N\kn} |\widehat{\Psi}(1)|\right)\\
\ll~&\sum_{\kn} \N\kn|b(\kn)|^2\left(\frac{\kappa_K}{V(z)} + \frac{h_H M' e^{O(n_K)}}{\N\kn} \right). 
\end{align*}
by \cref{lem:WeightFunction}. 
Since $b(\kn)$ is supported on prime ideals whose norm is greater than $y$, the above is
\begin{equation}
\ll \sum_{\kp} \N\kp |b(\kp)|^2\left(\frac{\kappa_K}{V(z)} + \frac{h_H M z^{2+2\epsilon} e^{O_{\epsilon}(n_K)}}{y} \right)
\label{eqn:LargeSieve_Penultimate}
\end{equation}
as $M' = M z^{2+2\epsilon}$ with $M$ defined by \eqref{def:M}. Now, select $z$ satisfying
\begin{equation}
y = h_H M e^{B_1 n_K} \cdot  z^{2+4\epsilon}, 
\label{eqn:LargeSieve_zChoice}
\end{equation}
where $B_1 = B_1(\epsilon) > 0$ is sufficiently large. From \eqref{eqn:LargeSieve_yRange}, it follows that $1 \leq z \leq y$ and further,
\[
z \geq e^{B_2 n_K} (n_K^{n_K} D_K)^{\tfrac{1}{2}+\epsilon}
\]
where $B_2 = B_2(\epsilon) > 0$ is sufficiently large. Hence, after inputting this choice of $z$ into \eqref{eqn:LargeSieve_Penultimate}, it follows by \cref{cor:harmonic_sum} that
\begin{align*}
\sum_{\chi(H)=1}\int_{-T}^{T}\left|\sum_{\kp}b(\kp)\chi(\kp)\N\kp^{-it}\right|^2 dt
\ll_{\epsilon}  \frac{1}{\log z} \sum_{\kp} \N\kp |b(\kp)|^2.
\end{align*}
Finally, from \eqref{eqn:LargeSieve_yRange} and \eqref{eqn:LargeSieve_zChoice}, one can verify that $\log z \gg_{\epsilon} \log y$ which completes the proof after rescaling $\epsilon > 0$ appropriately. 
\hfill \qed

\section{Detecting the Zeros of Hecke $L$-functions}
\label{sec:ZeroDetector}
\subsection{Setup}

The objective of this section is to prove \cref{prop:ZeroDetector} so we fix some notation to be used throughout this section. Let $H \pmod{\kq}$ be a congruence class group and let $\chi \pmod{\kq}$ be a Hecke character, satisfying $\chi(H) = 1$,  induced from the primitive character $\chi^* \pmod{\kf_{\chi}}$. Define $Q=Q_H$ by \eqref{def:Q}, and for $T \geq 1$, 
\begin{equation}
\cL := 2 \log D_K + \log Q + n_K \log(T+3) + \Theta n_K
\label{def:cL}
\end{equation}
where $\Theta \geq 1$ is sufficiently large. Let $R \geq 1$ be sufficiently large and $0 < r_0 < \tfrac{1}{16}$.  Suppose $\tau \in \R$ and $r > 0$ satisfy
\begin{equation}
\frac{R}{\cL} \leq r < r_0 \quad \text{and} \quad  |\tau| \leq T. 
\end{equation}
Assume $L(s,\chi)$ has a non-trivial zero $\rho$ satisfying
\begin{equation}
|1+i\tau-\rho| \leq r
\label{DetectedZero}. 
\end{equation}

The proof of \cref{prop:ZeroDetector} is divided into two main steps, with the final arguments culminating in \cref{ZeroDetector_Proof}.  The final arguments critically hinge on the following power sum estimate due to Kolesnik and Straus \cite{Kolesnik-Straus}. 
\begin{thm}
\label{KS-PowerSum} For any integer $M \geq 0$ and complex numbers $z_1,\dots,z_N$, there is an integer $k$ with $M+1 \leq k \leq M+N$ such that
\[
|z_1^k + \cdots + z_N^k| \geq 1.007 \Big( \frac{N}{4e(M+N)} \Big)^N |z_1|^k. 
\]
\end{thm}
\begin{remark}
	One can verify that the 	expression $\big( \frac{N}{4e(M+N)} \big)^N$ is a decreasing function of $N$. 
\end{remark}

Any improvement on the constant $4e$ in \cref{KS-PowerSum} would lead to a reduction of the exponent $73.2$ in \cref{prop:ZeroDetector}, but $4e$ has been shown by Makai \cite{Makai} to be best possible.

\subsection{A Large Derivative}
 
 Denote
\begin{equation}
F(s) := \frac{L'}{L}(s,\chi^*)
\label{BigDerivative-F}
\end{equation}
and $\xi := 1 + r + i\tau$. Using \cref{KS-PowerSum}, the goal of this subsection is to show $F(s)$ has a large high order derivative, which we establish in the following lemma. 
 
 \begin{lem} \label{BigDerivative} Keeping the above notation, if $\epsilon \in (0,1/4)$ and $r_0 < \epsilon/3.8$ then
 \[
\delta(\chi) \cdot \mathbf{1}_{\{ |\tau| < 4 r\}}(\tau) +\Big|   \frac{(-1)^k r^{k+1} }{k!} \cdot  F^{(k)}(\xi)  \Big| \gg \frac{ \exp(-16.6 \cdot \phi r \cL) }{2^{k+1}}
 \]
where $\phi = 1 + \tfrac{4}{\pi} \epsilon + 16\epsilon^2$ and for some integer $k$ satisfying 
 \begin{equation}
 25.0 \cdot \phi r \cL   \leq k \leq 28.8 \cdot \phi r \cL
\label{BigDerivative_kRange}
 \end{equation}
 \end{lem}
 \begin{proof} By \cite[Lemma 1.10]{Weiss}, 
 \[
 F(s) + \frac{\delta(\chi)}{s-1}  = \sum_{|1+i\tau - \rho| < 1/2 } \frac{1}{s-\rho} + G(s)
 \]
uniformly in the region
\[
|1+i\tau - s| < 1/2,
 \]
where $G(s)$ is analytic and $|G(s)| \ll \cL$ in this region. Differentiating the above formula $k$ times and evaluating at $\xi = 1+r+i\tau$, we deduce
\begin{equation}
\label{F_kDerivatives}
\frac{(-1)^k}{k!} \cdot  F^{(k)}(\xi) + \frac{\delta(\chi)}{(\xi-1)^{k+1}} =   \sum_{|1+i\tau - \rho| < 1/2} \frac{1}{(\xi-\rho)^{k+1}}  +  O(4^k \cL)
\end{equation}
for $\eta > 0$ and $0 < r < r_0 < 1/8$. The error term arises from bounding $G^{(k)}(\xi)$ using Cauchy's integral formula with a circle of radius of $1/4$. 

Let  $A \geq 1$ be a fixed absolute parameter to be specified later. For zeros $\rho$ satisfying $Ar < |1+i\tau - \rho| < 1/2$ in \eqref{F_kDerivatives}, notice
\[
(A^2+1) r^2 < r^2 + |1+i\tau - \rho|^2 \leq |\xi-\rho|^2 \leq (r + |1+i\tau-\rho|)^2 \leq (r+1/2)^2 < 1.
\]
Denoting $A_1 = \sqrt{A^2+1} \geq 2$, it follows by partial summation that 
\begin{align*}
	\sum_{Ar < |1+i\tau-\rho| < 1/2 } \frac{1}{|\xi-\rho|^{k+1}} 
	& \leq \int_{A_1 r}^{1} u^{-k-1} dN_{\chi}(u; \xi ) \\
	& =  (k+1) \int_{A_1 r}^{1} \frac{N_{\chi}(u; \xi)}{u^{k+2}} du + O(\cL)
\end{align*}
where we bounded $N_{\chi}(1; \xi) \ll \cL$ using \cite[Lemma 2.2]{LMO} and recalling $r\cL \geq R \gg 1$. 
By \cref{ZerosInCircle-Classical}, the above is therefore
\begin{equation}
\begin{aligned}
& \leq (k+1)  \int_{A_1 r}^{\infty}  \frac{2u \cL +4 + 4\delta(\chi)}{u^{k+2}} du  + O(\cL )  \\
& \leq \frac{2A_1 r  \cL+ 4 + 4\delta(\chi)}{(A_1 r)^{k+1}} + \int_{A_1 r}^{\infty} \frac{2\cL}{u^{k+1}} du + O(\cL)  \\
& \leq \frac{ 2 \{ 1+ \tfrac{1}{k} \} A_1 r  \cL+ 4 + 4\delta(\chi)}{(A_1 r)^{k+1}} + O(\cL) \\
& \ll \frac{r \cL}{(A_1 r)^{k+1}} .
\end{aligned}
\label{BigDerivative_FarAwayZeros}
\end{equation}
 By considering cases, one may bound the $\delta(\chi)$-term in \eqref{F_kDerivatives} as follows:
\begin{equation}
r^{k+1} \cdot \Big| \frac{ \delta(\chi) }{(\xi-1)^{k+1}} \Big| \leq \delta(\chi) \cdot \mathbf{1}_{\{ |\tau| < Ar \}}(\tau) + \frac{1}{A_1^{k+1}} 
\label{BigDerivative_Principal}
\end{equation}
where $\mathbf{1}$ is an indicator function. Combining \eqref{F_kDerivatives}, \eqref{BigDerivative_FarAwayZeros} and the above yields
\begin{equation}
\begin{aligned}
& \delta(\chi) \mathbf{1}\{ |\tau| < Ar \}  + \Big| \frac{(-1)^k r^{k+1}}{k!} \cdot  F^{(k)}(\xi)  \Big|  \\
& \qquad \geq \Big| \sum_{|1+i\tau - \rho| \leq Ar} \frac{1}{(\xi-\rho)^{k+1}} \Big| \cdot r^{k+1} - O\Big(  \frac{r \cL}{A_1^{k+1}} + (4r)^{k+1} \cL\Big). 
\end{aligned}
\label{BigDerivative_UglyError}
\end{equation}
To lower bound the remaining sum over zeros, we wish to apply \cref{KS-PowerSum}. Denote
\[
N = N_{\chi}(Ar; 1+i\tau) = \#\{ \rho : L(\rho, \chi) = 0, |1+i\tau - \rho| \leq Ar\}.
\]
Let $\epsilon \in (0,1/4)$ be fixed. Provided
\begin{equation}
r_0 < \frac{\epsilon}{A}
\label{BigDerivative_r0_Cond1}
\end{equation}
then by \cref{ZerosInCircle-Convexity} and the definition of $\cL$ in \eqref{def:cL} it follows that
\begin{equation}
N \leq \phi  Ar \cL + 4 + 4\delta(\chi) 
\label{BigDerivative-Nsize}
\end{equation}
as $\Theta$ is sufficiently large (depending on $\epsilon$). 
We require a choice of $M$ which depends on the fixed absolute parameters $\alpha \in (0,1), \epsilon \in (0,1/4)$ and $A \geq 1$, all of which will be specified later. Define
\begin{equation}
M := \Big\lceil  \frac{ \phi Ar \cL + 4 + 4\delta(\chi) }{\alpha} \Big\rceil 
\label{BigDerivative_M}
\end{equation}
so $N \leq \alpha M$ by \eqref{BigDerivative-Nsize}. Thus, from \cref{KS-PowerSum} and  \eqref{DetectedZero},  
\begin{equation}
\Big| \sum_{|1+i\tau-\rho| \leq Ar} \frac{1}{(\xi-\rho)^{k+1}}  \Big| \geq \Big( \frac{\alpha}{4e(1+\alpha)} \Big)^{\alpha M} \frac{1}{(2r)^{k+1} } 
\label{BigDerivative_CloseZeros}
\end{equation}
for some $M + 1 \leq k \leq (1+\alpha)M$.  To simplify the error term in \eqref{BigDerivative_UglyError}, notice $r\cL \ll M \ll k$ so 
\[
(4r)^{k+1} \cL \ll k (4r)^{k} \ll kA_1^{-k}. 
\]
provided
\begin{equation}
r_0 < \frac{1}{4A_1}.
\label{BigDerivative_r0_Cond2}
\end{equation}
Moreover, select $A \geq 1$ so that $A_1 = \sqrt{A^2+1}$ is given by
\begin{equation}
A_1 = 2 \Big( \frac{4e(1+\alpha)}{\alpha} \Big)^{\alpha}  (1+ \eta)
\label{BigDerivative_A}
\end{equation} 
where $\eta \in (0,1)$ is fixed. This choice implies
\[
A_1^{-(k+1)} \leq \Big( \frac{\alpha}{4e(1+\alpha)} \Big)^{\alpha M}   \frac{1}{2^{k+1}(1+\eta)^{k+1}}
\]
since $\alpha k \geq \alpha M$. Incorporating \eqref{BigDerivative_CloseZeros} and the subsequent observations into \eqref{BigDerivative_UglyError} yields
\begin{equation}
\begin{aligned}
& \delta(\chi) \mathbf{1}\{ |\tau| < Ar \}  + \Big| \frac{(-1)^k r^{k+1}}{k!} \cdot  F^{(k)}(\xi)  \Big|  \\
& \qquad\qquad \geq \Big( \frac{\alpha}{4e(1+\alpha)} \Big)^{\phi Ar \cL + 8} \cdot  \frac{1}{2^{k+1}} \Big\{ 1 -  O\Big( \frac{k}{(1+\eta)^{k+1}} \Big) \Big\}   
\end{aligned}
 \label{BigDerivative_Penultimate}
\end{equation}
after bounding $N$ by \eqref{BigDerivative-Nsize} and assuming \eqref{BigDerivative_r0_Cond1} and \eqref{BigDerivative_r0_Cond2} hold.  Since $k \gg M \gg r\cL \gg R$, we may impose $R$ to be sufficiently large, depending on $\eta \in (0,1)$, so that the above error term is negligible. Finally, we select $\alpha = 0.15$ and $\eta = 10^{-4}$ yielding $A = 3.752\dots$ by \eqref{BigDerivative_A}. With these choices,  conditions \eqref{BigDerivative_r0_Cond1} and \eqref{BigDerivative_r0_Cond2} are automatically satisfied as $r_0 < \epsilon/3.8 < 1/16$ by assumption.  The desired result follows after inputting these values into \eqref{BigDerivative_Penultimate} and recalling $M+1 \leq k \leq (1+\alpha)M$.  
\end{proof}
\begin{rem*} Let us motivate our choice of $\alpha = 0.15$. Ultimately, we will wish to maximize the righthand side of \eqref{BigDerivative_Penultimate} when $k$ is large; that is, supposing 
\[
k \approx (1+\alpha)M \approx(1+\alpha) \frac{A}{\alpha}  \cdot  \phi r \cL
\]
by \eqref{BigDerivative_M}. By \eqref{BigDerivative_A}, notice $A \approx \sqrt{4C_{\alpha}^2-1}$ for $\eta \in (0,1)$ sufficiently small and where $C_{\alpha} = \big( \frac{4e(1+\alpha)}{\alpha} \big)^{\alpha}$. Therefore, we select $\alpha \in (0,1)$ which minimizes the quantity 
\[
\frac{ \sqrt{4C_{\alpha}^2-1}}{\alpha} \Big( \log C_{\alpha} +(1+\alpha)  \log 2 \Big)
\]
and this turns out to be roughly $\alpha = 0.15$. 
\end{rem*}

\subsection{Short Sum over Prime Ideals}
Defining $\Lambda_K$ by \eqref{def:vonMangoldt}, it follows by the Euler produt for $L(s,\chi^*)$ that 
\[
F(s) = \frac{L'}{L}(s,\chi^*)  = -\sum_{\kn} \chi^*(\kn)  \Lambda_K(\kn) (\N\kn)^{-s} 
\]
for $\Re\{s\} > 1$. Differentiating the above formula $k$ times, we deduce
\begin{equation}
\frac{(-1)^{k+1} r^{k+1}}{k!} \cdot  F^{(k)}(\xi) =   \sum_{\kn} \frac{ \Lambda_K(\kn) \chi^*(\kn)}{\N\kn^{1+i\tau}}  \cdot r E_k(r \log \N\kn)
\label{F_kDerivative_Lseries}
\end{equation}
 for any integer $k \geq 1$, where $\xi = 1 + r + i\tau$ and 
\begin{equation}
 E_k(u) =  \frac{u^k e^{-u}}{k!}. 
 \label{def:E_k}
\end{equation}
As a preliminary observation, notice from Stirling's formula in the form
\begin{align*}
k^k e^{-k} \sqrt{2\pi k} \leq k! \leq k^k e^{-k} \sqrt{2\pi k} e^{1/12k}
\end{align*}
(see \cite{DLMF}), one can verify
\begin{equation}
E_k(u) \leq 
\begin{cases}
(1+\eta)^{-k} & \text{if $u \leq \dfrac{k}{e(1+\eta)}$}, \\
(1+\eta)^{-k} e^{-\delta u} & \text{if $u \geq \tfrac{2}{1-\delta} \log\big( \tfrac{2(1+\eta)}{1-\delta} \big) k$},
\end{cases}
\label{E_k-Bounds}
\end{equation}
for $k \geq 1, \eta > 0$ and $\delta \in (0,1)$. The goal of this subsection is to bound the infinite sum in \eqref{F_kDerivative_Lseries} by an integral average of short sums over prime ideals. 
\begin{lem} \label{ShortPrimeSum} Keeping the above notation, assume the integer $k$ satisfies \eqref{BigDerivative_kRange}. Then 
\[
\Big| \sum_{\kn} \frac{\chi^*(\kn) \Lambda_K(\kn)}{\N\kn^{1+i\tau}}  \cdot r E_k(r \log \N\kn) \Big| \leq  r^2 \int_y^x \Big| \sum_{\substack{y \leq \N\kp < u}} \frac{\chi^*(\kp) \log \N \kp  }{\N\kp^{1+i\tau}}  \Big| \frac{du}{u} + O\big( e^{-16.8 \phi r \cL} (2.01)^{-k} \big)
 \]
 provided $x,y \geq 1$ satisfy
 \begin{equation}
\log y \leq 2.3 \phi \cL \quad \text{ and } \quad 122\phi \cL \leq  \log x \ll \cL.  
\label{ShortPrimeSum_Range}
 \end{equation}
 \end{lem}
\begin{proof} First, divide the sum on the LHS of into four sums:
\begin{align*}
\sum_{\kn} & = \sum_{\N\kp < y} + \sum_{y \leq \N\kp < x} + \sum_{\N\kp \geq x} + \sum_{\kn \text{ not prime} }  \\
& = S_1 + S_2 + S_3 + S_4,
\end{align*}
say. It suffices to show
\begin{align*}
|S_2| & \leq r^2 \int_y^x \Big| \sum_{\substack{y \leq \N\kp < u}} \frac{\chi(\kp) \log \N \kp  }{\N\kp^{1+i\tau}}  \Big| \frac{du}{u}  + O( (3.95)^{-k}), \\
|S_j| &  \ll (3.95)^{-k} \quad \text{ for $j=1,3,4,$}
\end{align*}
because, by \eqref{BigDerivative_kRange},
\begin{align*}
(3.95)^{-k} = e^{-k \log( 3.95/2.01)} (2.01)^{-k} \ll  e^{- 25.0 \log(3.95/2.01) \phi r \cL} (2.01)^{-k} \ll e^{-16.8 \phi r \cL} (2.01)^{-k}.
\end{align*}
Thus, we shall repeatedly use \eqref{E_k-Bounds} with $\eta = 3$ and $\delta = 0.01$. These choices, along with \eqref{BigDerivative_kRange} and \eqref{ShortPrimeSum_Range}, imply that 
\begin{equation}
E_k(r \log \N\kn)  \leq
\begin{cases}
4^{-k} & \text{if $\N\kn \leq y$}, \\
4^{-k} (\N\kn)^{-0.01 r} & \text{if $\N\kn \geq x$}.
\end{cases}
\label{E_k-Bounds_Simple}
\end{equation}
Hence, for $S_1$, observe by \cite[Lemma 1.11]{Weiss} that
\[
|S_1| \leq  r 4^{-k} \sum_{\substack{\N\kp < y}} \frac{\log \N\kp }{\N\kp}  \ll r 4^{-k}  \log( y D_K) \ll  r\cL 4^{-k} \ll k 4^{-k} \ll (3.95)^{-k}
\]
as $k \gg r\cL \gg R$ by \eqref{BigDerivative_kRange} and $R$ is sufficiently large. Similarly, for $S_3$, we use \cite[Lemma 1.11]{Weiss}, \cref{lem:PrimeSum}, and \eqref{ShortPrimeSum_Range} to deduce \\
\begin{equation*}
\begin{aligned}
|S_3| & \leq  r4^{-k} \sum_{\substack{\N\kp \geq x}} \frac{\log \N\kp}{(\N\kp)^{1+ 0.01 r}}  \\
& \leq r 4^{-k} \Big( -\frac{\zeta_K'}{\zeta_K}(1+0.01 r) - \sum_{\N\kp < x} \frac{\log \N\kp}{(\N\kp)^{1+0.01 r} } \Big) \\
& \ll r 4^{-k} \Big( r^{-1}+ \log D_K + \log(D_K x) \Big)    \\
& \ll 4^{-k} + r \cL 4^{-k} \\
& \ll (3.95)^{-k}. 
\end{aligned}
\end{equation*}
For $S_4$, since $\sum_{k=0}^{\infty} E_k(u) = 1$, observe
\[
E_k(r \log \N\kn) = (2r)^k (\N\kn)^{1/2-r} E_k(\tfrac{1}{2}\log \N\kn) \leq 4^{-k} (\N\kn)^{1/2-r}
\]
as $r < r_0 \leq 1/8$ by assumption. Thus, by \cite[Lemma 1.11]{Weiss}, 
\begin{align*}
|S_4| & \leq r \sum_{\kp} \sum_{m \geq 2} \frac{\log \N\kp}{(\N\kp^m)} E_k(r \log \N\kp^m) \\
& \leq 4^{-k} r \sum_{\kp} \sum_{m \geq 2} \frac{\log \N\kp}{(\N\kp^m)^{1/2+r} }  \\
& \ll 4^{-k} r \big( r^{-1} + \log D_K \big) \\
&  \ll  k 4^{-k} \\
& \ll (3.95)^{-k}.
\end{align*}
Finally, for the main term $S_2$, define
\[
W(u) = W_{\chi}(u; \tau) := \sum_{y \leq \N\kp < u} \frac{\chi(\kp) \log \N \kp  }{\N\kp^{1+i\tau}},
\]
so by partial summation
\begin{equation}
S_2    =  r W(x) E_k(r \log x) - r^2 \int_y^x  W(u) E_k'(r \log u) \frac{du}{u}
\label{ShortPrimeSum_S2}
\end{equation}
as $W(y) = 0$.  From \eqref{E_k-Bounds_Simple}, notice
\[
|r W(x) E_k(r \log x)| \ll r 4^{-k} x^{-0.01 r}  \sum_{y \leq \N\kp < x} \frac{\Lambda_K(\kn)}{\N\kn}  \ll 4^{-k}   r \log x \ll k 4^{-k} \ll (3.95)^{-k}
\]
by \cite[Lemma 1.11]{Weiss} and \eqref{ShortPrimeSum_Range}. One can  verify that $|E_k'(u)| = |E_{k-1}(u)  - E_k(u)| \leq E_{k-1}(u) + E_k(u) \leq 1$ from definition \eqref{def:E_k} so the desired estimate  follows from \eqref{ShortPrimeSum_S2}. 
\end{proof}

\subsection{Proof of \cref{prop:ZeroDetector}} \label{ZeroDetector_Proof}
 From  \eqref{F_kDerivative_Lseries} and \cref{BigDerivative,ShortPrimeSum}, it follows that 
\begin{equation}
\exp(-36.6 \phi r \cL)
\ll r^2 \int_y^x \Big| \sum_{\substack{y \leq \N\kp < u}} \frac{\chi^*(\kp) \log \N \kp  }{\N\kp^{1+i\tau}}  \Big| \frac{du}{u}  + \delta(\chi) \mathbf{1}_{\{|\tau| < 4r\}}(\tau)
\label{ZeroDetector-1}
\end{equation}
since
\[
\frac{ \exp(-16.6 \phi r \cL) }{2^{k+1}} \gg \exp( - (16.6 \phi r \cL + k \log 2) ) \gg \exp(-36.6 \phi r \cL) 
\]
for $k$ satisfying \eqref{BigDerivative_kRange}. As $y > e^{\cL} \geq \N\kf_{\chi}$, it follows $\chi^*(\kp) = \chi(\kp)$ for $y \leq \N\kp < x$ so we may replace $\chi^*$ with $\chi$ in \eqref{ZeroDetector-1}. Squaring both sides of \eqref{ZeroDetector-1}, replacing $\chi^*$ with $\chi$, and applying Cauchy-Schwarz gives the desired result upon noting $\int_y^x \frac{du}{u} \ll \log(x/y) \ll \cL$ by \eqref{ShortPrimeSum_Range}. 
\hfill \qed
 
\section{Zero Repulsion: The Deuring-Heilbronn Phenomenon}
\label{sec:DH_proof}

To prove \cref{DH-MainTheorem} and establish Deuring-Heilbronn phenomenon for $L$-functions of ray class characters, we will critically use the following power sum inequality. 

\begin{thm}[Lagarias-Montgomery-Odlyzko] \label{LMO-PowerSum} Let $\epsilon > 0$ and a sequence of complex numbers $\{z_n\}_n$ be given. Let $s_m = \sum_{n=1}^{\infty} z_n^m$ and suppose that $|z_n| \leq |z_1|$ for all $n \geq 1$. Define
\begin{equation}
M := \frac{1}{|z_1|}\sum_{n} |z_n|. 
\label{PowerSum-M}
\end{equation}
Then there exists $m_0$ with $1 \leq m_0 \leq (12+\epsilon) M$ such that
\[
\Re\{ s_{m_0} \} \geq \frac{\epsilon}{48+5\epsilon} |z_1|^{m_0}. 
\]
\end{thm}
\begin{proof} This is a modified version of \cite[Theorem 4.2]{LMO}; see \cite[Theorem 2.3]{Zaman_2015c} for details. 
\end{proof}
We prepare for the application of this result by establishing a few preliminary estimates and then end this section with the proof of \cref{DH-MainTheorem}.

\subsection{Preliminaries} 
 
  \begin{lem} \label{DH-TrivialZeroSum} Let $\chi \pmod{\kq}$ be a Hecke character. For $\sigma \geq 1$ and $t \in \R$, 
\begin{align*}
& \sum_{\omega \, \mathrm{ trivial}} \frac{1}{|\sigma+it-\omega|^2}  \leq  \begin{cases} \big( \frac{1}{2\sigma} + \frac{1}{\sigma^2} \big)  \cdot n_K & \text{if $\chi$ is primitive,} \\
\big( \frac{1}{2\sigma} + \frac{1}{\sigma^2} \big) \cdot n_K + \big(  \frac{1}{2 \sigma} + \frac{2}{\sigma^2 \log 2}  \big) \cdot  \log \N\kq   & \text{unconditionally}, 
\end{cases}
\end{align*}
where the sum is over all trivial zeros $\omega$ of $L(s,\chi)$ counted with multiplicity.
\end{lem}
\begin{proof} Suppose $\chi \pmod{\kq}$ is induced by the primitive character $\chi^* \pmod{\kf_{\chi}}$. Then 
\[
L(s,\chi) = P(s, \chi) L(s,\chi^*) \quad \text{where } \quad P(s, \chi) = \prod_{\substack{ \kp \mid \kq \\ \kp \nmid \kf_{\chi}} } \Big(1 - \frac{\chi^*(\kp)}{\N\kp^{s}} \Big)
\]
for all $s \in \mathbb{C}$. Thus, the trivial zeros of $L(s,\chi)$ are either zeros of the finite Euler product $P(s,\chi)$ or trivial zeros of $L(s,\chi^*)$. We consider each separately. From \eqref{TrivialZeros} and \eqref{GammaFactor_Exponents}, observe
\begin{align*}
 \sum_{\substack{ \omega \, \mathrm{ trivial} \\ L(\omega, \chi^*) = 0} } \frac{1}{|\sigma+it-\omega|^2} 
 & \leq a(\chi) \sum_{k=0}^{\infty}  \frac{1}{(\sigma+2k)^2 + t^2}  + b(\chi) \sum_{k=0}^{\infty} \frac{1}{(\sigma+2k+1)^2 +t^2} \\
 & \leq n_K \sum_{k=0}^{\infty}  \frac{1}{(\sigma+2k)^2} \leq \Big( \frac{1}{2\sigma} + \frac{1}{\sigma^2} \Big) n_K
\end{align*}
Now, if $\chi$ is primitive then $P(s,\chi) \equiv 1$ and hence never vanishes. Otherwise, notice the zeros of each $\kp$-factor in the Euler product of $P(s,\chi)$ are totally imaginary and are given by
\[
 a_{\chi}(\kp) i + \frac{2\pi i \Z}{\log \N\kp}
\]
for some $0 \leq a_{\chi}(\kp)  < 2\pi/\log \N\kp$.  Translating these zeros $\omega \mapsto \omega + it$ amounts to choosing another representative $0 \leq b_{\chi}(\kp; t) < 2\pi /\log\N\kp$. Therefore, 
\begin{align*}
 \sum_{\substack{ \omega \, \mathrm{ trivial} \\ P(\omega, \chi) = 0} } \frac{1}{|\sigma+it-\omega|^2} 
 & \leq   2 \sum_{ \substack{ \kp \mid \kq \\ \kp \nmid \kf_{\chi} } } \sum_{k=0}^{\infty}  \frac{1}{\sigma^2 + (2\pi k/\log \N\kp)^2} \\
 & \leq   2 \sum_{ \substack{ \kp \mid \kq \\ \kp \nmid \kf_{\chi} } }\Big( \frac{1}{\sigma^2} +  \int_0^{\infty} \frac{1}{\sigma^2 + (2\pi x/\log \N\kp)^2 } dx \Big) \\
  & \leq  2 \sum_{ \substack{ \kp \mid \kq \\ \kp \nmid \kf_{\chi} } }\Big(  \frac{\log \N\kp}{4 \sigma} +  \frac{1}{\sigma^2} \Big) \\
  & \leq   \Big(\frac{1}{2 \sigma} +  \frac{2}{\sigma^2 \log 2}  \Big) \log \N\kq,
\end{align*}
as required.  
\end{proof}

 \begin{lem} \label{DH-ZeroSum} Suppose $\psi \pmod{\kq}$ is real and $\chi \pmod{\kq}$ is arbitrary. For $\sigma = \alpha+1$ with $\alpha \geq 1$ and $t \in \R$, 
\begin{align*}
& \sum_{\substack{\rho \\  \zeta_K(\rho) = 0}} \frac{1}{|\sigma-\rho|^2}  + \sum_{\substack{\rho \\  L(\rho, \psi) = 0}}  \frac{1}{|\sigma -\rho|^2}    + \sum_{\substack{\rho \\  L(\rho, \chi) = 0}}  \frac{1}{|\sigma + it -\rho|^2}  + \sum_{\substack{\rho \\  L(\rho, \psi\chi) = 0}}  \frac{1}{|\sigma +it -\rho|^2} \\
& \qquad \leq \frac{1}{\alpha} \cdot \Big[ \frac{1}{2}\log(D_K^3 \N\kq^2 D_{\psi} )  +\Big( \log(\alpha+2) + \frac{2}{\alpha+1} - 2\log \pi \Big) n_K \\
& \qquad\qquad\qquad +  n_K \log( \alpha+2+|t|) + \frac{4}{\alpha} + \frac{4}{\alpha+1} \Big],  
\end{align*}
where the sums are over all non-trivial zeros  of the corresponding $L$-functions. 
\end{lem}
\begin{rem*} If $\psi$ is trivial, notice that the LHS equals 
\[
2  \Big( \sum_{\substack{\rho \\  \zeta_K(\rho) = 0}} \frac{1}{|\sigma-\rho|^2} +  \sum_{\substack{\rho \\  L(\rho, \chi) = 0}}  \frac{1}{|\sigma + it -\rho|^2} \Big). 
\]
This additional factor of $2$ will be useful to us later. 
\end{rem*}

\begin{proof}
Suppose $\psi$ and $\chi$ are induced from the primitive characters $\psi^*$ and $\chi^*$ respectively. From the identity
\[
0 \leq ( 1 + \psi^*(\kn) )(1 + \Re\{ \chi^*(\kn) (\N\kn)^{-it} \} ),
\]
it follows that 
\[
0 \leq - \Re\Big\{ \frac{\zeta_K'}{\zeta_K}(\sigma) +  \frac{L'}{L}(\sigma, \psi^*) + \frac{L'}{L}(\sigma+it, \chi^*) + \frac{L'}{L}(\sigma+it, \psi^*\chi^*)  \Big\}.
\]
Applying \cref{ExplicitFormula,digamma} to each term yields
\begin{equation}
\begin{aligned}
0 & \leq \tfrac{1}{2}\log(D_K^3 \N\kq^2 D_{\psi} ) +  n_K \log(\sigma+1+|t|) + (\log(\sigma+1) + 2\sigma^{-1} - 2\log \pi) n_K \\
& \qquad + \frac{1+\delta(\psi)}{\alpha} + \frac{1+\delta(\psi)}{\alpha+1} + \Re\Big\{ \frac{\delta(\chi) + \delta(\chi \psi) }{\alpha+it} + \frac{\delta(\chi) + \delta(\chi \psi)}{\alpha+1+it} \Big\}  \\
&  - \Re\Big\{ \sum_{\substack{\rho \\  \zeta_K(\rho) = 0}} \frac{1}{\sigma-\rho}  + \sum_{\substack{\rho \\  L(\rho, \psi) = 0}}  \frac{1}{\sigma -\rho}    + \sum_{\substack{\rho \\  L(\rho, \chi) = 0}}  \frac{1}{\sigma + it -\rho}  + \sum_{\substack{\rho \\  L(\rho, \psi\chi) = 0}}  \frac{1}{\sigma +it -\rho} \Big\}\\
\end{aligned}
\label{BoundZeroSum}
\end{equation}
Since $0 < \beta < 1$, we notice
\[
\Re\Big\{ \frac{1}{\sigma+it-\rho} \Big\} = \frac{\alpha+1-\beta}{|\sigma+it-\rho|^2} \geq \frac{\alpha}{|\sigma+it-\rho|^2} 
\]
and
\[
 \Re\Big\{ \frac{1}{\alpha+it} + \frac{1}{\alpha+1+it} \Big\}  \leq  \frac{1}{\alpha} + \frac{1}{\alpha+1}. 
\]
Rearranging \eqref{BoundZeroSum} and employing these observations gives the desired conclusion. 
\end{proof}

\subsection{Proof of \cref{DH-MainTheorem}}

We divide the proof according to whether $\psi$ is quadratic or trivial. The arguments in each case are similar but require some minor differences. 

\subsubsection{$\psi$ is quadratic.}
\label{DH-MainTheorem_Quadratic}
Let $m$ be a positive integer, $\alpha \geq 1$ and $\sigma = \alpha+1$.  From the identity
\[
0 \leq (1+\psi^*(\kn) )(1+ \Re\{ \chi^*(\kn) (\N\kn)^{-i\gamma'} \} )
\]
and \cref{ExplicitFormula_HigherDerivatives} with $s = \sigma+i\gamma'$, it follows 
\begin{equation}
\Re\Big\{ \sum_{n=1}^{\infty}  z_n^m \Big\} \leq \frac{1}{\alpha^m} - \frac{1}{(\alpha+1-\beta_1)^{2m}} + \Re\Big\{ \frac{\delta(\chi)+\delta(\psi \chi)}{(\alpha+i\gamma')^{2m}} - \frac{\delta(\chi)+\delta(\psi\chi)}{(\alpha +1+i\gamma' -\beta_1)^{2m}} \Big\} 
\label{DH-Quadratic_PowerSumPrep}
\end{equation}
where $z_n = z_n(\gamma')$ satisfies $|z_1| \geq |z_2| \geq \dots$ and runs over the multisets
\begin{equation}
\begin{aligned}
& \{  (\sigma-\omega)^{-2} :  \omega \text{ is any zero of $\zeta_K(s)$} \}, \\
& \{  (\sigma-\omega)^{-2}  :  \omega \neq \beta_1 \text{ is any zero of $L(s, \psi^*)$} \}, \\
& \{  (\sigma+i\gamma'-\omega)^{-2}  :  \omega \neq \beta_1 \text{ is any zero of $L(s, \chi^*)$} \}, \\
& \{  (\sigma+i\gamma'-\omega)^{-2}  :  \omega \neq \beta_1 \text{ is any zero of $L(s, \psi^*\chi^*)$} \}. \\
\end{aligned}
\label{DH-Quadratic_z_n}
\end{equation}
Note that the multisets includes trivial zeros of the corresponding $L$-functions and $\psi^*\chi^*$ is a Hecke character (not necessarily primitive) modulo the least common multiple of $\kf_{\chi}$ and $\kf_{\psi}$. With this choice, it follows
\begin{equation}
(\alpha+1/2)^{-2} \leq (\alpha+1-\beta')^{-2} \leq |z_1| \leq \alpha^{-2}. 
\label{DH-zProperty}
\end{equation}
The RHS of  \eqref{DH-Quadratic_PowerSumPrep} may be bounded via the observation
\[
\Big| \frac{1}{(\alpha+it)^{2m}} - \frac{1}{(\alpha + it + 1-\beta_1)^{2m}} \Big| \leq \alpha^{-2m} \Big| 1 - \frac{1}{(1 + \frac{1-\beta_1}{\alpha+it})^{2m} } \Big|   \ll \alpha^{-2m-1} m(1-\beta_1),
\]
whence
\begin{equation}
\Re\Big\{ \sum_{n=1}^{\infty}  z_n^m \Big\} \ll \alpha^{-2m-1} m(1-\beta_1). 
\label{DH-PowerSum_RHS}
\end{equation}
On the other hand, by \cref{LMO-PowerSum}, for $\epsilon > 0$, there exists some $m_0 = m_0(\epsilon)$ with $1 \leq m_0 \leq (12+\epsilon)M$ such that
\[
\Re\Big\{ \sum_{n=1}^{\infty} z_n^{m_0} \Big\} \geq \tfrac{\epsilon}{50}  |z_1|^{m_0} \geq \tfrac{\epsilon}{50} (\alpha+1-\beta')^{-2m_0} \geq \tfrac{\epsilon}{50} \alpha^{-2m_0} \exp(-\tfrac{2m_0}{\alpha}(1-\beta') ), 
\]
where $M= |z_1|^{-1} \sum_{n=1}^{\infty} |z_n|$. Comparing with \eqref{DH-PowerSum_RHS} for $m = m_0$, it follows that
\begin{equation}
\exp(-(24+2\epsilon)\tfrac{M}{\alpha} (1-\beta') ) \ll_{\epsilon} \tfrac{M}{\alpha} (1-\beta_1). 
\label{DH-Penultimate}
\end{equation}
Therefore, it  suffices to bound $M/\alpha$ and optimize over $\alpha \geq 1$.

By \eqref{DH-Quadratic_z_n}, the quantity $M$ is a sum involving non-trivial and trivial zeros of certain $L$-functions. For the non-trivial zeros, we employ \cref{DH-ZeroSum} with $D_{\psi} \leq D_K \N\kq$ since $\psi$ is quadratic. For the trivial zeros, apply \cref{DH-TrivialZeroSum}   in the ``primitive" case for $\zeta_K(s), L(s,\psi^*), L(s,\chi^*)$ and  in the ``unconditional" case for $L(s,\psi^*\chi^*)$. Then from \eqref{DH-zProperty}, it follows that
\begin{equation}
\label{DH-M_Quadratic}
\begin{aligned}
\frac{M}{\alpha}  & \leq \frac{(\alpha+1/2)^{2}}{\alpha^2}   \cdot \Big[ 2 \log D_K + \Big( \frac{3}{2} + \frac{\alpha}{2\alpha+2} + \frac{2 \alpha}{(\alpha+1)^2 \log 2} \Big) \log \N\kq \\
& \qquad\qquad +\Big( \log(\alpha+2)  + 2 - 2\log \pi + \frac{4\alpha}{(\alpha+1)^2} \Big) n_K \\
& \qquad\qquad+  n_K \log(\alpha+2+ T)+ \frac{4}{\alpha} + \frac{4}{\alpha+1} \Big],  
\end{aligned}
\end{equation}
for $\alpha \geq 1$. Selecting $\alpha = 18$, inputting the resulting bound into \eqref{DH-Penultimate}, and fixing $\epsilon > 0$ sufficiently small completes the proof for $\psi$ quadratic.  

\begin{rem*}
	The final choice of $\alpha$ was somewhat arbitrary because the coefficients of $\log D_K, \log \N\kq$ and $n_K$ in \eqref{DH-M_Quadratic} cannot be simultaneously minimized. As $\alpha \rightarrow \infty$, it is apparent that the coefficients of $\log D_K$ and $\log \N\kq$ decrease and converge to a minimum  but the coefficient of $n_K$ grows arbitrarily large. Hence, in the interest of having relatively small coefficients of comparable size for all quantities, we chose the value $\alpha = 18$. 
\end{rem*}

\subsubsection{$\psi$ is trivial.}
Now, for $\psi$ trivial, we begin with the identity
\[
0 \leq 1 + \Re\{\chi^*(\kn) (\N\kn)^{-i\gamma'} \}.
\]
This similarly implies
\begin{equation}
\Re\Big\{ \sum_{n=1}^{\infty}  z_n^m \Big\} \leq \frac{1}{\alpha^m} - \frac{1}{(\alpha+1-\beta_1)^{2m}} + \Re\Big\{ \frac{\delta(\chi)}{(\alpha+i\gamma')^{2m}} - \frac{\delta(\chi)}{(\alpha +1+i\gamma' -\beta_1)^{2m}} \Big\} 
\label{DH-AllZeros_PowerSumPrep}
\end{equation}
for a new choice $z_n = z_n(\gamma')$ satisfying $|z_1| \geq |z_2| \geq \dots$ and which runs over the multisets
\begin{equation}
\begin{aligned}
& \{  (\sigma-\omega)^{-2} :  \omega \neq \beta_1 \text{ is any zero of $\zeta_K(s)$} \}, \\
& \{  (\sigma+i\gamma'-\omega)^{-2}  :  \omega \neq \beta_1 \text{ is any zero of $L(s, \chi^*)$} \}. \\
\end{aligned}
\label{DH-AllZeros_z_n_Principal}
\end{equation}
Following the same arguments as before, we may arrive at \eqref{DH-Penultimate} for the new quantity $M = |z_1|^{-1} \sum_{n=1}^{\infty} |z_n|$. To bound the non-trivial zeros arising in $M$, apply \cref{DH-ZeroSum} with $D_{\psi} = D_K$ since $\psi$ is trivial.  For the trivial zeros, apply \cref{DH-TrivialZeroSum}  in the ``primitive" case for both $\zeta_K(s)$ and $L(s,\chi^*)$. It follows from \eqref{DH-zProperty} that
\begin{equation}
\label{DH-M_Principal}
\begin{aligned}
\frac{M}{\alpha}  & \leq \frac{(\alpha+1/2)^{2}}{\alpha^2}   \cdot \Big[ \log D_K + \frac{1}{2} \log \N\kq \\
& \qquad\qquad +\Big( \frac{1}{2} \log(\alpha+2)  + 1 - \log \pi - \frac{1}{\alpha+1}   \Big) n_K \\
& \qquad\qquad+  \frac{1}{2} n_K \log\Big(\alpha+2+T\Big)+ \frac{2}{\alpha} + \frac{2}{\alpha+1} \Big]. 
\end{aligned}
\end{equation}
As with the previous case, selecting $\alpha = 18$, inputting the resulting bound into \eqref{DH-Penultimate}, and fixing $\epsilon > 0$ sufficiently small yields the desired result. \hfill \qed

\bibliographystyle{abbrv}
\bibliography{LFZD}

\begin{thebibliography}{10}

\bibitem{BS}
E.~Bach and J.~Sorenson.
\newblock Explicit bounds for primes in residue classes.
\newblock {\em Math. Comp.}, 65(216):1717--1735, 1996.

\bibitem{Bombieri}
E.~Bombieri.
\newblock {\em Le grand crible dans la th\'eorie analytique des nombres}.
\newblock Soci\'et\'e Math\'ematique de France, Paris, 1974.
\newblock Avec une sommaire en anglais, Ast{\'e}risque, No. 18.

\bibitem{Burgess}
D.~A. Burgess.
\newblock The distribution of quadratic residues and non-residues.
\newblock {\em Mathematika}, 4:106--112, 1957.

\bibitem{Chen}
J.~R. Chen.
\newblock The exceptional set of {G}oldbach numbers. {II}.
\newblock {\em Sci. Sinica Ser. A}, 26(7):714--731, 1983.

\bibitem{Deuring}
M.~Deuring.
\newblock \"{U}ber den {T}schebotareffschen {D}ichtigkeitssatz.
\newblock {\em Math. Ann.}, 110(1):414--415, 1935.

\bibitem{Fogels}
E.~Fogels.
\newblock On the zeros of {H}ecke's {$L$}-functions. {I}, {II}.
\newblock {\em Acta Arith.}, 7:87--106, 131--147, 1961/1962.

\bibitem{Gallagher}
P.~X. Gallagher.
\newblock A large sieve density estimate near {$\sigma =1$}.
\newblock {\em Invent. Math.}, 11:329--339, 1970.

\bibitem{GS}
E.~S. Golod and I.~R. {\v{S}}afarevi{\v{c}}.
\newblock On the class field tower.
\newblock {\em Izv. Akad. Nauk SSSR Ser. Mat.}, 28:261--272, 1964.

\bibitem{Graham1}
S.~W. Graham.
\newblock {\em Applications of sieve methods}.
\newblock ProQuest LLC, Ann Arbor, MI, 1977.
\newblock Thesis (Ph.D.)--University of Michigan.

\bibitem{HBLinnik}
D.~R. Heath-Brown.
\newblock Zero-free regions for {D}irichlet {$L$}-functions, and the least
  prime in an arithmetic progression.
\newblock {\em Proc. London Math. Soc. (3)}, 64(2):265--338, 1992.

\bibitem{Heilbronn}
H.~Heilbronn.
\newblock Zeta-functions and {$L$}-functions.
\newblock In {\em Algebraic {N}umber {T}heory ({P}roc. {I}nstructional {C}onf.,
  {B}righton, 1965)}, pages 204--230. Thompson, Washington, D.C., 1967.

\bibitem{Huxley}
M.~N. Huxley.
\newblock Large values of {D}irichlet polynomials. {III}.
\newblock {\em Acta Arith.}, 26(4):435--444, 1974/75.

\bibitem{Ihara}
Y.~Ihara.
\newblock On the {E}uler-{K}ronecker constants of global fields and primes with
  small norms.
\newblock In {\em Algebraic geometry and number theory}, volume 253 of {\em
  Progr. Math.}, pages 407--451. Birkh\"auser Boston, Boston, MA, 2006.

\bibitem{Ihara2}
Y.~Ihara.
\newblock The {E}uler-{K}ronecker invariants in various families of global
  fields.
\newblock In {\em Arithmetics, geometry, and coding theory ({AGCT} 2005)},
  volume~21 of {\em S\'emin. Congr.}, pages 79--102. Soc. Math. France, Paris,
  2010.

\bibitem{IK}
H.~Iwaniec and E.~Kowalski.
\newblock {\em Analytic number theory}, volume~53 of {\em American Mathematical
  Society Colloquium Publications}.
\newblock American Mathematical Society, Providence, RI, 2004.

\bibitem{Jutila}
M.~Jutila.
\newblock On {L}innik's constant.
\newblock {\em Math. Scand.}, 41(1):45--62, 1977.

\bibitem{KadiriNg}
H.~Kadiri and N.~Ng.
\newblock Explicit zero density theorems for {D}edekind zeta functions.
\newblock {\em J. Number Theory}, 132(4):748--775, 2012.

\bibitem{Kolesnik-Straus}
G.~Kolesnik and E.~G. Straus.
\newblock On the sum of powers of complex numbers.
\newblock In {\em Studies in pure mathematics}, pages 427--442. Birkh\"auser,
  Basel, 1983.

\bibitem{LMO}
J.~C. Lagarias, H.~L. Montgomery, and A.~M. Odlyzko.
\newblock A bound for the least prime ideal in the {C}hebotarev density
  theorem.
\newblock {\em Invent. Math.}, 54(3):271--296, 1979.

\bibitem{LO}
J.~C. Lagarias and A.~M. Odlyzko.
\newblock Effective versions of the {C}hebotarev density theorem.
\newblock In {\em Algebraic number fields: {$L$}-functions and {G}alois
  properties ({P}roc. {S}ympos., {U}niv. {D}urham, {D}urham, 1975)}, pages
  409--464. Academic Press, London, 1977.

\bibitem{LLS}
Y.~Lamzouri, X.~Li, and K.~Soundararajan.
\newblock Conditional bounds for the least quadratic non-residue and related
  problems.
\newblock {\em Math. Comp.}, 84(295):2391--2412, 2015.

\bibitem{RJLO_JT}
R.~Lemke~Oliver and J.~Thorner.
\newblock Effective log-free zero density estimates for automorphic
  {$L$}-functions and the {S}ato-{T}ate conjecture.
\newblock submitted, \url{http://arxiv.org/abs/1505.03122}.

\bibitem{Linnik}
U.~V. Linnik.
\newblock On the least prime in an arithmetic progression. {II}. {T}he
  {D}euring-{H}eilbronn phenomenon.
\newblock {\em Rec. Math. [Mat. Sbornik] N.S.}, 15(57):347--368, 1944.

\bibitem{Makai}
E.~Makai.
\newblock On a minimum problem. {II}.
\newblock {\em Acta Math. Acad. Sci. Hungar.}, 15:63--66, 1964.

\bibitem{milneCFT}
J.~Milne.
\newblock {\em Class Field Theory (v4.02)}.
\newblock 2013.
\newblock Available at www.jmilne.org/math/.

\bibitem{Murty}
V.~K. Murty.
\newblock The {E}uler-{K}ronecker constant of a cyclotomic field.
\newblock {\em Ann. Sci. Math. Qu\'ebec}, 35(2):239--247, 2011.

\bibitem{DLMF}
N.~I. of~Standards and Technology.
\newblock Digital library of mathematical functions.
\newblock Version 1.0.5, October 1, 2012.

\bibitem{OS}
K.~Ono and K.~Soundararajan.
\newblock Ramanujan's ternary quadratic form.
\newblock {\em Invent. Math.}, 130(3):415--454, 1997.

\bibitem{Rademacher}
H.~Rademacher.
\newblock On the {P}hragm\'en-{L}indel\"of theorem and some applications.
\newblock {\em Math. Z}, 72:192--204, 1959/1960.

\bibitem{Weiss}
A.~Weiss.
\newblock The least prime ideal.
\newblock {\em J. Reine Angew. Math.}, 338:56--94, 1983.

\bibitem{Weiss_Thesis}
A.~R. Weiss.
\newblock {\em The least prime ideal with prescribed decomposition behaviour}.
\newblock ProQuest LLC, Ann Arbor, MI, 1980.
\newblock Thesis (Ph.D.)--The Ohio State University.

\bibitem{Xylouris}
T.~Xylouris.
\newblock On the least prime in an arithmetic progression and estimates for the
  zeros of {D}irichlet {$L$}-functions.
\newblock {\em Acta Arith.}, 150(1):65--91, 2011.

\bibitem{Zaman3}
A.~Zaman.
\newblock Explicit bounds on the least prime ideal.
\newblock in preparation.

\bibitem{Zaman_2015a}
A.~Zaman.
\newblock Explicit estimates for the zeros of {H}ecke ${L}$-functions.
\newblock submitted, \url{http://arxiv.org/abs/1502.05679}.

\bibitem{Zaman_2015c}
A.~Zaman.
\newblock Bounding the least prime ideal in the {C}hebotarev density theorem.
\newblock submitted, \url{http://arxiv.org/abs/1508.00287}.

\end{thebibliography}
\end{document}